\documentclass[11pt,reqno,oneside, a4paper]{amsart}
\usepackage[ansinew]{inputenc}
\usepackage[english]{babel}
\usepackage[T1]{fontenc} 
\usepackage{lmodern}
\usepackage{amsfonts} 		
\usepackage{amsmath}
\usepackage{amssymb}
\usepackage{bbm}
\usepackage{amsthm}
\usepackage{arydshln}
\usepackage{xcolor}
\usepackage{enumitem}
\usepackage{mathabx}
\usepackage{comment}
\usepackage{mathrsfs}
\usepackage{url}
\usepackage{hyperref}
\usepackage{cite}

\usepackage[capitalize]{cleveref}
\crefname{enumi}{}{}
\crefname{equation}{}{}

\usepackage{geometry}

 \specialcomment{supplementary}
{\colorlet{savedcolor}{.} \color{blue} \begingroup \ttfamily \bigskip \smallskip  \noindent \underline{Supplementary details:} \newline \newline \footnotesize }{\endgroup \smallskip \color{savedcolor}}
 \excludecomment{supplementary}

\setenumerate[0]{label=\textnormal{(\roman*)}}

\newcommand{\R}{\mathbb{R}}
\newcommand{\RD}{\mathbb{R}^d}
\newcommand{\EX}{\mathbb{E}}

\numberwithin{equation}{section}

\newcommand{\N}{\mathbb{N}}

\newcommand{\PR}{\mathbb{P}}

\newcommand{\norm}[1]{\ensuremath{\Vert #1 \Vert}}
\newcommand{\scnorm}[1]{\ensuremath{\left\Vert #1 \right\Vert}}
\newcommand{\abs}[1]{\ensuremath{\vert #1 \vert}}

\newcommand{\eps}{\varepsilon}

\newtheorem{thm}{Theorem}[section]
\newtheorem{definition}[thm]{Definition}
\newtheorem{lemma}[thm]{Lemma}
\newtheorem{cor}[thm]{Corollary}
\newtheorem{prop}[thm]{Proposition}
\theoremstyle{remark}

\newtheorem{remark}[thm]{Remark}

\begin{document}

\title[1D stochastic pressure equation]{1D stochastic pressure equation with log-correlated Gaussian coefficients}


\author[Avelin]{Benny Avelin}
\address{Uppsala University, Department of Mathematics, 751 06 Uppsala, Sweden}
\email{benny.avelin@math.uu.se}

\author[Kuusi]{Tuomo Kuusi}
\address{University of Helsinki, Department of Mathematics and Statistics, PO Box 68 (Pietari Kalmin katu 5), 00014 University of Helsinki, Finland}
\email{tuomo.kuusi@helsinki.fi}

\author[Nummi]{Patrik Nummi}
\address{Aalto University School of Electrical Engineering, Department of Information and Communications Engineering, PO Box \textcolor{black}{15600}, 00076 Aalto, Finland}
\email{patrik.nummi@aalto.fi}

\author[Saksman]{Eero Saksman}
\address{University of Helsinki, Department of Mathematics and Statistics, PO Box 68 (Pietari Kalmin katu 5), 00014 University of Helsinki, Finland}

\email{eero.saksman@helsinki.fi}

\author[T\"olle]{Jonas M. T\"olle}
\address{Aalto University School of Science, Department of Mathematics and Systems Analysis, PO Box 11100, 00076 Aalto, Finland}
\email{jonas.tolle@aalto.fi}

\author[Viitasaari]{Lauri Viitasaari}
\address{Aalto University School of Business, Department of Information and Service Management, PO Box \textcolor{black}{21210}, 00076 Aalto, Finland}
\email{lauri.viitasaari@aalto.fi}

\date{\today}

\thanks{BA acknowledge support by the Swedish Research Council (VR) dnr: 2019--04098. TK and PN acknowledge support by the Academy of Finland and the European Research Council (ERC) under the European Union's Horizon 2020 research and innovation program (grant agreement No 818437). JMT was supported by the Academy of Finland and the European Research Council (ERC) under the European Union's Horizon 2020 research and innovation program (grant agreements no. 741487 and no. 818437). TK and ES acknowledge support by the Academy of Finland CoE FiRST}

\begin{abstract}
    We study unique solvability for one-dimensional stochastic pressure equation with diffusion coefficient given by the Wick exponential of log-correlated Gaussian fields. We prove well-posedness for Dirichlet, Neumann and periodic boundary data and the initial value problem, covering the cases of both the Wick renormalization of the diffusion and of point-wise multiplication. We provide explicit representations for the solutions in both cases, characterized by the $S$-transform and the Gaussian multiplicative chaos measure.
\end{abstract}

\allowdisplaybreaks


\maketitle

{\footnotesize\tableofcontents}



\section{Introduction}
We are studying existence and uniqueness of solutions $U$ to the following second order linear stochastic boundary value problem on $[0,T]\subset\R$,
\begin{align}\label{eq:ode}
    \begin{cases}
        -(e^{\diamond (-X_\beta(t))} \times U'(t))' & = f(t),  \qquad t \in (0,T), \\
        \textnormal{boundary data at}               & \{0,T\},
    \end{cases}
\end{align}
with deterministic $f\in L^1(0,T)$. The diffusion coefficient $e^{\diamond (-X_\beta)}$ is formally the Wick-exponential of a log-correlated scalar Gaussian random field \textcolor{black}{(see Definition \ref{def:log-correlated-field})} $X_\beta$ on $[0,T]$ with parameter $\beta\in (0,1)$, and here the symbol $\times$ denotes either the usual point-wise product $\cdot$ in $\R$ or the Wick product $\diamond$. As now the diffusion coefficient $e^{\diamond (-X_\beta)}$ cannot be defined point-wise, we approximate $X_\beta$ by its mollified version and define the solution to~\cref{eq:ode} as the limit (in a suitable sense) of the solutions of the approximate equations
\begin{align*}
    \begin{cases}
        -C_\varepsilon(e^{\diamond (-X_{\beta,\varepsilon}(t))} \times U'_\varepsilon(t))' & =  f(t),  \qquad t \in (0,T), \\
        \textnormal{boundary data at}                                                      & \{0,T\},
    \end{cases}
\end{align*}
where now $C_\varepsilon$ is a suitably chosen normalization constant, depending on whether $\times$ corresponds to the usual product in which case $C_\varepsilon = e^{\EX [X_{\beta,\varepsilon}^2]}$, or whether $\times$ corresponds to the Wick product in which case $C_\varepsilon \equiv 1$. We note that the normalization affects the boundary data, but is omitted here.

\textcolor{black}{The present work considers the one-dimensional case of the stochastic boundary value problem studied in \cite{renormalized}. In that paper, the problem was analyzed in the $d$-dimensional setting but was restricted to the equation involving the Wick product and periodic boundary data. Here, we specialize to $d=1$, but generalize the analysis to include both the Wick product and the more physically direct, but mathematically more singular, ordinary scalar product. This one-dimensional focus makes the problem tractable, allowing for a detailed analysis of both cases. In particular, we are able to derive explicit solution formulas for a variety of boundary conditions, which enables a direct comparison between the regularized "non-Wick" solutions and the distributional "Wick" solutions.}

The problem is the one-dimensional stochastic pressure equation which can be used to model enhanced geothermal heating. The solutions to the stochastic pressure equation are steady state solutions to a creeping-water flow in porous media. The Wick exponential of $X_\beta$, i.e.~$e^{\diamond (-X_\beta)}$ models the random permeability of the rock. For $X_\beta$ replaced by white noise, the stochastic pressure equation has been studied with methods from white noise theory in~\cite{holden2009,holden1995,holden1993,Lind1991}, see also~\cite{galvis2009,wan2013,zhang2017} for the numerical analysis of these equations. In particular, the one-dimensional problem with white noise has been discussed in~\cite[Subsection 4.6.3]{holden2009}. A common ingredient in showing existence and uniqueness to these types of equations is to employ some kind of integral transform (see, for instance~\cite{holden1995,potthoff1992} and references therein). 

Observing problem~\cref{eq:ode} closely, one finds that it can be written formally as the ordinary differential equation (ODE),
\begin{equation*}
    U''\textcolor{black}{-} X_\beta'\times U'= -\frac{1}{e^{\diamond (-X_\beta)}} \times f, 
\end{equation*}
which is generally ill-posed due to the lack of regularity for the realizations of $X_\beta$. Note that $U$ is formally the stationary solution to the (1+1)-dimensional weighted random Kolmogorov equation with singular drift, compare with~\cite{Russo2007,flandoli2017},
\begin{equation*}
    -\frac{1}{e^{\diamond (-X_\beta)}}\times\partial_t u=\partial_x\partial_x u \textcolor{black}{-} \partial_x X_\beta \times \partial_x u+\frac{1}{e^{\diamond (-X_\beta)}}\times Cf, \quad t>0,\quad x\in (0,T).
\end{equation*}
The need of renormalization for stochastic partial differential equations of this type is already natural in one spatial dimension. See~\cite{H:14,BHZ:19,BCCH:21} for a systematic treatment of well-posedness and renormalization of singular stochastic PDEs.

For our stationary problem, we are treating both the point-wise product and the so-called Wick product, cf.~\cite{gjessing1991,janson1997,kondratiev1996}. See~\cite[Subsection 3.5]{holden2009} for a direct comparison in the one-dimensional white noise case. From a certain asymptotic point of view, Wick multiplication corresponds to the It\^o-stochastic integral and point-wise multiplication corresponds to the Stratonovich stochastic integral, see~\cite{Hu1996} for details. Wick products are the first order term in the so-called Wick-Malliavin expansion of the product of two random variables~\cite{venturi2013}, so that in particular, Wick products preserve the expected value. We note that in the case of log-correlated fields, due to the lack of second moments, the white noise approach has to be transferred carefully using so-called Gaussian multiplicative chaos measures~\cite{saksman2018,saksman2015}.

As boundary data for~\cref{eq:ode}, we are studying the cases of Dirichlet, Neumann and periodic boundary data, as well as the initial value problem (IVP).
The different versions of boundary conditions and Wick vs.~non-Wick for our problem are presented in the following~\cref{table:equations}. The corresponding explicit representations for the solutions are gathered in~\cref{table:solutions}. 

\begin{table}[ht!]
    \begin{tabular}{lll}
        \hline
        \multicolumn{1}{|l|}{Dirichlet} & \multicolumn{1}{l|}{Wick}& \multicolumn{1}{l|}{Non-Wick}
        \\
        \hline
        \multicolumn{1}{|l|}{Equation} & \multicolumn{1}{l|}{$-(e^{\diamond (-X_\beta)} \diamond U')'=f $}    & \multicolumn{1}{l|}{$-e^{\EX [X_\beta^2]}(e^{\diamond(- X_\beta)} \cdot U')'=f$}
        \\
        \hline
        \multicolumn{1}{|l|}{Data} 
        & \multicolumn{1}{l|}{
            \begin{tabular}[c]{@{}l@{}}
                $U(0)=U_1$,
                \\ 
                $U(T)=U_2$
            \end{tabular}
        } 
        & \multicolumn{1}{l|}{
            \begin{tabular}[c]{@{}l@{}}
                $U(0)=U_1$,
                \\ 
                $U(T)=U_2$
            \end{tabular}
        }
        \\
        \hline
        \hline
        \multicolumn{1}{|l|}{Neumann}   & \multicolumn{1}{l|}{Wick}& \multicolumn{1}{l|}{Non-Wick}
        \\
        \hline
        \multicolumn{1}{|l|}{Equation}  & \multicolumn{1}{l|}{$-(e^{\diamond (-X_\beta)} \diamond U')'=f$} & \multicolumn{1}{l|}{$-e^{\EX [X_\beta^2]}(e^{\diamond (-X_\beta)}\cdot U')'=f$}
        \\
        \hline
        \multicolumn{1}{|l|}{Data}      
        & \multicolumn{1}{l|}{
            \begin{tabular}[c]{@{}l@{}}
                $e^{\diamond X_\beta(0)} \diamond U'(0)=U_1$,
                \\ 
                $e^{\diamond X_\beta(T)} \diamond U'(T)=U_2$
            \end{tabular}
        } 
        & \multicolumn{1}{l|}{
            \begin{tabular}[c]{@{}l@{}}
                $e^{\EX [X_\beta^2]}e^{\diamond (-X_\beta(0))}\cdot U'(0)= U_1$,
                \\ 
                $e^{\EX [X_\beta^2]} e^{\diamond (-X_\beta(T))}\cdot U'(T) = U_2$
            \end{tabular}
        }
        \\
        \hline
        \hline
        \multicolumn{1}{|l|}{IVP} & \multicolumn{1}{l|}{Wick}& \multicolumn{1}{l|}{Non-Wick}
        \\
        \hline
        \multicolumn{1}{|l|}{Equation} & \multicolumn{1}{l|}{$-(e^{\diamond (-X_\beta)} \diamond U')'=f$} & \multicolumn{1}{l|}{$-e^{\EX [X_\beta^2]}(e^{\diamond(- X_\beta)}\cdot U')'=f $}
        \\
        \hline
        \multicolumn{1}{|l|}{Data}      
        & \multicolumn{1}{l|}{
            \begin{tabular}[c]{@{}l@{}}
                $U(0)=U_1$,
                \\ 
                $e^{\diamond X_\beta(0)}\diamond U'(0)=U_2$
            \end{tabular}
        }                      
        & \multicolumn{1}{l|}{
            \begin{tabular}[c]{@{}l@{}}
                $U(0)=U_1$, 
                \\ 
                $e^{\EX [X_\beta^2]}e^{\diamond(- X_\beta(0))}\cdot U'(0)=U_2$
            \end{tabular}
        }
        \\
        \hline
        \hline
        \multicolumn{1}{|l|}{Periodic} & \multicolumn{1}{l|}{Wick} & \multicolumn{1}{l|}{Non-Wick}
        \\
        \hline
        \multicolumn{1}{|l|}{Equation} & \multicolumn{1}{l|}{$-(e^{\diamond (-X_\beta)} \diamond U')'=f$} & \multicolumn{1}{l|}{$-e^{\EX [X_\beta^2]}(e^{\diamond (-X_\beta)} \cdot U')'=f$}
        \\
        \hline
        \multicolumn{1}{|l|}{Data}
        & \multicolumn{1}{l|}{
            \begin{tabular}[c]{@{}l@{}}
                $U(0)=U(T)$,
                \\
                $e^{\diamond X_\beta(0)} \diamond U'(0)=e^{\diamond X_\beta(T)} \diamond U'(T)$
            \end{tabular}
        }
        & \multicolumn{1}{l|}{
            \begin{tabular}[c]{@{}l@{}}
                $U(0)=U(T)$, \\ $e^{\diamond(- X_\beta(0))} \cdot U'(0) =e^{\diamond(- X_\beta(T))}\cdot U'(T)$
            \end{tabular}
        }
        \\
        \hline
    \end{tabular}
    \caption{Equations and corresponding boundary data to in Wick / non-Wick cases.}\label{table:equations}
\end{table}

\footnotesize
\begin{table}[ht!]
    \begin{tabular}{lll}
        \hline
        \multicolumn{1}{|l|}{Solution} & \multicolumn{1}{l|}{Dirichlet} & \multicolumn{1}{l|}{Random variable $\kappa$}
        \\
        \hline
        \multicolumn{1}{|l|}{Wick} & \multicolumn{1}{l|}{$U(t) = U_1 + \int_{0}^t (\kappa  + F(s)) \diamond e^{\diamond X_\beta(s)}\, ds$}
        & \multicolumn{1}{l|}{$\left(U_2 - U_1 - \int_{0}^T F(s) e^{\diamond X_\beta(s)}\, ds \right) \diamond \left(\int_{0}^T e^{\diamond X_\beta(s)}\, ds  \right)^{\diamond(-1)}$}
        \\
        \hline
        \multicolumn{1}{|l|}{Non-Wick}
        & \multicolumn{1}{l|}{$U(t) = U_1 + \int_{0}^t (\kappa + F(s)) e^{\diamond X_\beta(s)} \,ds$} 
        & \multicolumn{1}{l|}{$\left(U_2-U_1 - \int_{0}^T F(s) e^{\diamond X_\beta(s)} \, ds \right)\left(\int_{0}^T e^{\diamond X_\beta(s) }\, ds\right)^{-1}$}
        \\
        \hline
        \hline
        \multicolumn{1}{|l|}{Solution} & \multicolumn{1}{l|}{Neumann} & \multicolumn{1}{l|}{Random variable $\kappa$}
        \\
        \hline
        \multicolumn{1}{|l|}{Wick} 
        & \multicolumn{1}{l|}{
            \begin{tabular}[c]{@{}l@{}}
                $U(t)= \kappa +  \int_{0}^t (U_1 + F(s) ) e^{\diamond X(s)}\, ds$ 
                \\ 
            \end{tabular}
        } 
        & \multicolumn{1}{l|}{
            \begin{tabular}[c]{@{}l@{}}
                Arbitrary, instead compatibility 
                \\ 
                $U_2 - U_1 = -\int_{0}^T f(s) \, ds$
            \end{tabular}
        }
        \\
        \hline
        \multicolumn{1}{|l|}{Non-Wick} & \multicolumn{1}{l|}{$U(t) = \kappa + \int_{0}^t (U_1 + F(s)) e^{\diamond X(s)} \, ds$}
        & \multicolumn{1}{l|}{
            \begin{tabular}[c]{@{}l@{}}
                Arbitrary, instead compatibility 
                \\ 
                $U_2 - U_1 =  -\int_{0}^T f(s) \, ds$
            \end{tabular}
        }
        \\
        \hline
        \hline
        \multicolumn{1}{|l|}{Solution} & \multicolumn{1}{l|}{IVP} & \multicolumn{1}{l|}{Random variable $\kappa$}
        \\
        \hline
        \multicolumn{1}{|l|}{Wick} & \multicolumn{1}{l|}{$U(t) = U_1 + \int_{0}^t (U_2+ F(s)) e^{\diamond X_\beta(s)}\, ds$}& \multicolumn{1}{l|}{None}
        \\
        \hline
        \multicolumn{1}{|l|}{Non-Wick} & \multicolumn{1}{l|}{$U(t) = U_1 + \int_{0}^t (U_2 +F(s) )e^{\diamond X_\beta(s)} \, ds$} & \multicolumn{1}{l|}{None}
        \\
        \hline
        \hline
        \multicolumn{1}{|l|}{Solution $U(t)$} & \multicolumn{1}{l|}{Periodic} & \multicolumn{1}{l|}{Random variable $\kappa$}
        \\
        \hline
        \multicolumn{1}{|l|}{Wick} & \multicolumn{1}{l|}{$C + \int_{0}^t (\kappa + F(s))\diamond e^{\diamond X_\beta(s)} \,ds$} &
        \multicolumn{1}{l|}{
            \begin{tabular}[c]{@{}l@{}}
                $-\int_{0}^T F(s) e^{\diamond X_\beta(s) }\, ds \left(\int_{0}^T e^{\diamond X_\beta(s)}\,ds\right)^{\diamond(-1)}$,
                \\
                $f$ mean zero, $C$ arbitrary
            \end{tabular}
        }
        \\
        \hline
        \multicolumn{1}{|l|}{Non-Wick} & \multicolumn{1}{l|}{$C + \int_{0}^t (\kappa + F(s)) e^{\diamond X_\beta(s)} \,ds$} & \multicolumn{1}{l|}{$-\frac{\int_{0}^T F(s) e^{\diamond X_\beta(s) }\, ds}{\int_{0}^T e^{\diamond X_\beta(s)}\,ds}$, $f$ mean zero, $C$ arbitrary}
        \\
        \hline
    \end{tabular}
    \caption{Solutions to Wick / non-Wick problems with different boundary data. Here we denote $F(t) = -\int_{0}^t f(s) \, ds$ and $e^{\diamond X_\beta(s)}ds$ is the Gaussian multiplicative chaos measure $\mu_\beta(ds)$.}\label{table:solutions}
\end{table}
\normalsize

Note that in the solutions we have, with a slight abuse of notation, written $e^{\diamond X_\beta(s)}ds$ to refer to the Gaussian multiplicative chaos measure $d\mu_\beta(s)$ (see~\cref{sec:GMC} for details), which is not an absolutely continuous measure with respect to the Lebesgue measure on $[0,T]$. We obtain our solutions via an approximation argument. That is, we approximate the log-correlated field $X(s)$ and show that the solutions to the approximated equations converge, in a suitable sense. Actually, in \cref{table:equations} the normalization $e^{\EX [X_\beta^2]}$ blows up in the limit, as the variance of the log-correlated field is unbounded, while it is not true along the approximating sequence. However, it turns out that this is the correct normalization to obtain a non-trivial solutions in the limit. \textcolor{black}{Note also that the choice of coefficient $e^{\diamond (-X_\beta)}$ in Equation \eqref{eq:ode} translates into measures $e^{\diamond X_\beta(s)}ds$ on the solutions that corresponds to Gaussian multiplicative chaos. Similarly the coefficient  $e^{\diamond X_\beta}$ would translate into measures $e^{\diamond (-X_\beta(s))}ds$ on solutions. Since $-X_\beta$ and $X_\beta$ are equal in law, one observes that $e^{\diamond (-X_\beta(s))}ds$ are also well-defined.}

Note also that rather interestingly, the Wick and path-wise equations admit solutions that are of the same form, with the only difference being, whether products are considered as ordinary product or Wick product. This can be viewed as a consequence of the fact that by using the ordinary product, we insert multiplicative deterministic renormalizations, while this can be omitted in the Wick case as Wick product is renormalizing in itself. Particularly this is visible if one considers special cases of initial value problem or the Neumann problem, in which the solution to the path-wise equation and to the Wick equation are equal, well-defined H\"older continuous random functions. On the other hand, in the Dirichlet case, the solution to the Wick equation can be viewed only through abstract white noise analysis involving Wick products, while the solution to the path-wise equation can be again viewed as a
well-defined H\"older continuous random function.

Motivated by similarities in the representations, we also consider ``projections'' to the subspace of $L^2(\Omega)$ of the form, with suitably chosen $\beta'<1$,
\begin{align} \label{eq:u-ansatz}
    U(y) = \int_{(0,T)} \varphi(y,a) \, d \mu_{\beta'}(a),\qquad y \in (0,T)
\end{align}
where $\varphi(y,\cdot) \in W_0^{-\frac{1-\beta'^2}{2},2}([0,T])$, the dual of the fractional Sobolev (Hilbert) space $W^{\frac{1-\beta'^2}{2},2}([0,T])$. This corresponds to the projections into a subspace that have a similar form to what is presented in~\cref{table:solutions}, but having less moments if one chooses $\beta'$ larger. In particular, for $\beta'=\beta<1$ we recover the given solutions in the Neumann case or in the initial value problem. Our main contribution here is the definition of the so-called $S$-transform of random variables~\cref{eq:u-ansatz} at a point $X(z)$ of the log-correlated field (that do not admit point-wise evaluation in the classical sense) directly, allowing us to employ Fourier analysis techniques to invert a potential operator between a suitable Sobolev space and its dual.

The rest of the article is organized as follows. In~\cref{sec:preliminaries} we introduce some preliminaries. In particular, we recall basics of log-correlated fields and Gaussian multiplicative chaos measures. In~\cref{sec:non-wick} we prove, by means of approximation arguments, the existence of solutions to the renormalized path-wise equations. In~\cref{sec:wick} we use abstract Wick calculus to show existence of solutions to the Wick equations, again via approximation, and in~\cref{sec:projection} we study projections via $S$-transforms.


\section{Preliminaries}
\label{sec:preliminaries}
We begin by fixing some notation and by recalling some function spaces that will be used in the sequel. \textcolor{black}{First, we note that we will denote the value of a stochastic process $U$ at time $t$ by $U(t)$ (as opposed to the typical probabilistic convention $U_t$, or $U_t(\omega)$).} Throughout, we will consider an open bounded interval $D = (0,T) \subset \R$. We denote by $C^k(D)$ (or by $C^k$ in short) the space of functions $f \colon D \to \R$, such that $f$ is $k$-times continuously differentiable on $D$. Similarly, we denote $C^\infty(D)$ the space of infinitely continuously differentiable functions, that is, $C^\infty(\overline{D}) = \bigcap_{k=1}^\infty C^k(\overline{D})$. By $C_{\textcolor{black}{c}}^k$, and $C_{\textcolor{black}{c}}^\infty$ we denote the corresponding spaces with the additional requirement that its elements have compact support.
 \textcolor{black}{Throughout this paper,} E-valued random variables are assumed to be Bochner measurable (see~\cite[Chapter 2]{diestel-uhl-john1977} for details). In the sequel, we denote by $L^p(\Omega;E)$, for $p\geq 1$, the space of Bochner measurable $E$-valued random variables such that $\EX [\Vert Y\Vert_E^p] < \infty$.

\subsection{Log-correlated Gaussian field}\label{sec:GMC}
In this subsection we recall some basic facts on log-correlated Gaussian fields and Gaussian multiplicative chaos measures. For details, we refer to survey articles~\cite{duplantier2014,rhodes2014} and references therein.

\begin{definition}[Log-correlated Gaussian field]\label{def:log-correlated-field}
    A centered (distribution-valued) Gaussian process $X$ on a bounded open interval $D \subset \mathbb{R}$ is called  \emph{a log-correlated field}  if it has covariance
    \begin{equation}
        \label{eq:general-cov}
        R(y,z):= R_X(y,z) := \EX [X(y)X(z)] = \log \frac{1}{|z-y|} + g(z,y),
    \end{equation}
    where we assume $g(y,z) = h(y-z)$ for some $h \in \textcolor{black}{C_{c}^\infty(D)}$.
   
\end{definition}

We interpret the covariance yielding the kernel of the covariance operator, as point-wise evaluations are not bounded:
\begin{equation*}
    \EX [\langle X,\psi_1\rangle \langle X,\psi_2\rangle]\; =\; \int_{D\times D}R(y,z)\psi_1(y)\psi_2(z) \, dz \, dy
\end{equation*}
for any $\psi_1,\psi_2\in C_{\textcolor{black}{c}}^\infty (D)$.

\begin{remark}
 \textcolor{black}{The role of the function $g$ in~\cref{def:log-correlated-field} is to guarantee that~\cref{eq:general-cov} yields a well-defined covariance function. It is customary in the literature to assume $g \in \textcolor{black}{C_{c}^\infty(D \times D)}$. For our purposes however, \textcolor{black}{we assume, in addition, that $g$ depends only on} the distance between the points $z,y$, that is, $g(y,z) = h(y-z)$ for some smooth function $h$. This yields a scaling of the form $e^{R_X(y,z)} = e^{h(y-z)}\abs{y-z}^{-1}$ for the exponential covariance.
   We also note that in what follows, we will consider equations where the solution essentially inherits its regularity from the part of the covariance with the logarithmic singularity. Therefore, for the sake of \textcolor{black}{notational} simplicity, we typically present the computations only in the case $g\equiv 0$, while adding $g(y,z)=h(y-z)$ is straightforward.}
\end{remark}
\begin{definition}[Wick exponential]\label{def:wick-exponential}
    Let $Z$ be a centered Gaussian random variable. We define the Wick exponential of $Z$ by
    \begin{equation*}
        e^{\diamond Z} \;:=\; e^{Z - \EX [Z^2 /2]}.
    \end{equation*}
\end{definition}
\begin{remark}\label{rem:wick-exponential}
    Wick exponential can also be expressed as
    \begin{equation}\label{eq:wick-exponential-series}
        e^{\diamond Z} \;:=\; e^{Z - \EX [Z^2 /2]}\;=\;\sum_{n=0}^\infty \frac{Z^{\diamond n}}{n!},
    \end{equation}
    where $Z^{\diamond n}$ are Wick powers defined through Wick products, see~\cref{def:wick-prod}, and the convergence \textcolor{black}{of the sum} is in $L^2$.
\end{remark}
The family of Gaussian multiplicative chaos (=GMC) measures $d\mu_\beta$ on the real line, depending on a parameter $\beta \in (0, \sqrt{2})$, are random measures constructed from the Wick exponential of a log-correlated field $X$. The GMC measure on the real line is formally defined for all Borel sets $A \in \mathcal{B}(\R)$ as the random measure $\mu_\beta$ given by
\begin{align*}
    \mu_\beta(A) = \int_A e^{\beta X(y) - \frac{\beta^2}{2}\EX [X(y)^2] } \, dy,
\end{align*}
where $X$ is a log-correlated Gaussian field as in~\cref{def:log-correlated-field}. The fact that such a measure exists follows by the next lemma. If $\theta$ is a smooth\textcolor{black}{, positive definite} mollifier \textcolor{black}{(i.e. a compactly supported, smooth function on $(0,T)$ such that $\int_{0}^T \theta \, dx = 1$ and $\lim_{\eps \to 0} \eps^{-1} \theta(x/\eps) = \delta(x)$ where $\delta$ is the Dirac delta)} we denote by $X_\varepsilon$ the mollified field $X_\varepsilon = X \ast \theta_{\varepsilon}$, where $\theta_{\varepsilon}(y) = \varepsilon^{-1}\theta\left(\frac{y}{\varepsilon}\right)$ for $\eps > 0$. Then $X_\varepsilon$ defines a centered and smooth Gaussian field. \textcolor{black}{We refer to $X_\eps$ as the convolution approximation of the log-correlated Gaussian field.} We have the following useful technical result. \textcolor{black}{Parts (ii) and (iii) are known results taken from the literature while part (i) follows from elementary computations.}

\begin{lemma}[Approximation lemma]\label{lem:mollification-of-log-correlated-field}
    Let $X$ be a log-correlated Gaussian field on $(0,T)$, and $\beta \in (0,1)$ be a parameter. Let $0 < \eps < \eps'$, $x,y \in (0,T)$. Then the convolution approximation $X_\eps$ satisfies:
    \begin{enumerate}
        \item $X_\eps$ has continuous realizations and $\EX [X_\eps(z)^2]$ is finite and independent of $z$.
        \item The sequence of random measures
              \begin{align}\label{eq:wickexp}
                  d\mu_{\beta,\varepsilon}(y) := \mathbbm{1}_{(0,T)}(y)\exp\left(\beta X_\varepsilon(y) - \frac{\beta^2}{2}\EX \left[X_\varepsilon^2\right]\right)dy,
              \end{align}
              converges in probability in the space of Radon measures equipped with the topology of weak \textcolor{black}{star} convergence towards a random measure $d\mu_\beta(y)$ supported on $[0,T]$. \textcolor{black}{That is, let $f$ be a function in $C_b([0,T])$ and consider the family of random variables $\left\{\int_{0}^T f(y) \mu_{\beta,\varepsilon}(dy)   \right\}$, indexed by $\varepsilon > 0$. Then $$\int_{0}^T f(y) \mu_{\beta,\varepsilon}(dy) \to \int_{0}^T f(y) \mu_{\beta}(dy)$$ in probability as $\varepsilon \to 0$.} \textcolor{black}{The random measure $d\mu_\beta(y)$ is} called the Gaussian multiplicative chaos measure.
        \item $\lim_{\eps' \to 0} \EX \left[ X_\eps(x) X_{\eps'}(y)\right] = R(x,y)$, for $x \neq y$ fixed.
    \end{enumerate}
\end{lemma}
Before giving the proof, we remark that~\cref{def:log-correlated-field} can naturally be extended to open sets $D \subset \RD$. It can then be shown that~\cref{lem:mollification-of-log-correlated-field} holds for GMC measures on $\RD$ for $0 < \beta < \sqrt{2d}$.
\begin{proof}[Proof of~\cref{lem:mollification-of-log-correlated-field} ]
    The announced convergence in (ii)
    follows from standard results, see \textcolor{black}{Theorem 2.1 in \cite{robert2010} (see also \cite{kahane1985}).} Moreover,
    (iii) follows from~\cite[Lemma 2.8]{saksman2018}. For (i), it is well-known that, by definition via convolution, the field $X_\eps$ has continuous realizations. It remains to prove that the variance of $X_\eps(z)$ is finite and does not depend on the point $z$.
    Now $X_\eps(z) := (\theta_\eps \ast X)(z) = \int_{\R} \theta_\eps(z-y) X(y) \, dy$ which is centered. By the translational invariance of the Lebesgue measure we get\textcolor{black}{, in the case $g\equiv 0$,}
    \begin{align*}
        \EX \left[X_\eps(z)X_\eps(z')\right] & = \EX \left[ \int_{\R} \theta_\eps(z-y) X(y) \, dy \int_{\R} \theta_\eps(z'-x)X(x) \, dx   \right] \\
                                            & =
        \int_{\R} \int_{\R} \theta_\eps(z-y) \theta_\eps(z'-x) \EX [X(y) X(x)] \, dy \, dx                                                       \\
                                            & =
        \int_{\R} \int_{\R} \theta_\eps(z-y) \theta_\eps(z'-x) \log \frac{1}{\abs{x-y}} \, dy \, dx                                             \\
                                            & =
        \int_{\R} \int_{\R} \theta_\eps(a) \theta_\eps(b) \log \frac{1}{\abs{z'-z-(b-a)}} \, da \, db
    \end{align*}
    from which we see that the variance $\EX [X_\eps(z)^2]$ is finite and in fact does not depend on $z$. Moreover, it is straightforward to see that the argument carries through if one adds $g(x,z) = h(x-z)$. This completes the proof.
\end{proof}

\begin{supplementary}
    Just for clarity, proof of~\cref{eq:log-correlated-approximation-property}. By~\cite[Corollary 3.39]{janson1997}: if $\xi$ and $\eta$ have joint centered Gaussian distribution, then
    \begin{equation*}
        e^{\diamond (\xi + \eta)} = e^{-\EX [\xi \eta]} e^{\diamond \xi} e^{\diamond \eta}.
    \end{equation*}
    Choose $\xi = X$, and $\eta = -X$, the above identity then reduces to
    \begin{equation*}
        e^{\diamond 0} = e^{\EX [X^2]} e^{\diamond X} e^{\diamond (-X)}.
    \end{equation*}
    Since $e^{\diamond 0} =1$, the previous display becomes
    \begin{equation*}
        e^{-\EX [X^2]} e^{-\diamond X} = e^{\diamond (-X)}.
    \end{equation*}
    Since $X$ is symmetric, this is equal in law to
    \begin{equation*}
        e^{-\EX [X^2]} e^{-\diamond X} = e^{\diamond X},
    \end{equation*}
    which is precisely~\cref{eq:log-correlated-approximation-property}.
\end{supplementary}

For the reader's convenience, we record the fact that the GMC measure $\mu_\beta$ has positive $p$-moments up to order $2/\beta^2$, see   ~\cite[Theorem 2.11]{rhodes2014}\textcolor{black}{, and has all negative moments finite:}
\begin{thm}\label{thm:GMC-moments}
    Let $X$ be the log-correlated Gaussian field on an open set $D \subset \R$, and $d \mu_\beta$ the associated Gaussian multiplicative chaos measure with parameter $0<\beta<\sqrt{2}$. Then the measure $\mu_\beta$ admits finite $p$-moments for all $p \in (0,\frac{2}{\beta^2})$\textcolor{black}{, and all $p \in (-\infty, 0).$} In particular, $\EX [\mu_\beta(A)] = \abs{A}$ for any $A \in \mathcal{B}(D)$, where $\abs{A}$ denotes the Lebesgue measure of $A$.
\end{thm}

In the present paper we shall restrict ourselves to the so-called $L^2$-range, where $\beta <\sqrt{d}=1$, and this will be our standing assumption from now on.

Note that by our definitions we have that
\begin{equation*}
    d\mu_{\beta,\varepsilon}(y) = e^{\diamond X_{\beta,\varepsilon}(y)}dy,
\end{equation*}
where $e^{\diamond X_{\beta,\varepsilon}}$ is the Wick exponential of the Gaussian field $X_{\beta,\varepsilon} = \beta X_{\varepsilon}$. Thus, it is natural to make use of the notation
\begin{equation*}
    e^{\diamond X_\beta} \textcolor{black}{\, dy} := d\mu_\beta.
\end{equation*}
Since $X_\beta$ is centered (and hence $X_\beta$ and $-X_\beta$ have same distribution), we obtain that similarly $e^{\diamond (-X_\beta)}$ appearing formally in our equation corresponds to a chaos measure equivalent in law to $d\mu_\beta$.

\section{Unique solvability of non-Wick equation}\label{sec:non-wick}
In this section we present our main results concerning the problem with the scalar product. That is, we prove the existence and uniqueness of the solution to the problem
\begin{align} \label{eq:non-wick-equation-general-problem}
    \begin{cases}
        -e^{\EX [ X_\beta(t)^2]}(e^{\diamond (-X_\beta(t)\textcolor{black}{)}}\cdot U'(t)' & = f(t) \qquad t \in (0,T) \\
        \textnormal{boundary data}
    \end{cases}
\end{align}
and the representation formula of the solution. 
Above $X_\beta$ is the log-correlated Gaussian field with a parameter $\beta < 1$, $f \colon [0,T] \to \R$ is a deterministic, integrable function, \textcolor{black}{and $e^{\diamond (-X_\beta(t))}$ formally denotes the point mass of the GMC measure (constructed from $-X_\beta(t)$ via symmetry)}. As $X_\beta$ is a generalized random field with undefined variance, we proceed through approximation $X_{\beta,\varepsilon}$ \textcolor{black}{in the above equation,} and throughout (see, e.g.~\cref{thm:non-wick-eqn-is-solvable} below) the term $e^{\EX [ X_\beta(t)^2]}$ indicates a normalization $e^{\EX [ X_{\beta,\varepsilon}(t)^2]}$ in the approximating equation. More precisely, we will always interpret the solutions via the following definition.
\begin{definition}\label{defn:soln-of-non-wick-sde}
    Let $X_\eps$ be \textcolor{black}{the} convolution approximation of the log-correlated field $X$ on $[0,T]$ so that $d\mu_{\beta,\eps} \to d\mu_\beta$ in probability in the space of Radon measures as $\eps \to 0$. Let $U_\eps$ be the corresponding solution of the problem
    \begin{align*}
        \begin{cases}
            -e^{\EX [ X_{\beta,\varepsilon}(t)^2]}(e^{\diamond (-X_{\beta,\varepsilon}(t))}\cdot U_\varepsilon'(t))' & = f(t) \qquad t \in (0,T) \\
            \textnormal{boundary data for } U_\varepsilon.
        \end{cases}
    \end{align*}
    We say that the random process $U$ is a solution to~\cref{eq:non-wick-equation-general-problem} if $U_\eps \stackrel{P}{\to} U$ as $\eps \to 0$ in the space $C([0,T])$.
\end{definition}
\begin{remark}\label{rem:solution-uniqueness}
    Note that the solution $U_\eps$ is either unique or unique up to an additive constant, depending on the  boundary data used. \textcolor{black}{Note also that while a priori the definition does not provide a unique limiting solution $U$ since the limit of $U_\eps$ might depend on the chosen mollifier, it turns out that the solution is unique. Indeed, this follows from the fact that $d\mu_{\beta,\eps} \to d\mu_\beta$ independently of the chosen mollifier. }
\end{remark}

We are now in place to state the solvability result for problem~\cref{eq:non-wick-equation-general-problem}. Recall that below, all the boundary terms should be interpreted to hold for the approximating sequences and solutions. For example in the periodic case, condition $e^{\diamond (-X_\beta(0))}U'(0)=e^{\diamond (-X_\beta(T))} U'(T)$ is interpreted as $e^{\diamond (-X_{\beta,\varepsilon}(0))}U_\varepsilon'(0)=e^{\diamond (-X_{\beta,\varepsilon}(T))} U_\varepsilon'(T)$, for all $\varepsilon>0$.

\begin{thm}\label{thm:non-wick-eqn-is-solvable}
    Let $f \colon [0,T] \to \R$ be a deterministic, integrable function, and let $X_\beta=\beta X$, where $X$ is a log-correlated Gaussian field on $[0,T]$, and $\beta < 1$. Consider the equation
    \begin{equation*}
        -e^{\EX [X_\beta(t)^2]}(e^{\diamond (-X_\beta(t))}\cdot U'(t))' = f(t), \qquad t \in (0,T).
    \end{equation*}
    Denote $F(t) = -\int_0^t f(s) ds$, then the above equation with:
    \begin{enumerate}
        \item\label{item:non-wick:initial} \emph{Initial data}: $U(0)=U_1$ and $e^{\EX [X_\beta^2]}e^{\diamond (-X_\beta(0))}\cdot U'(0) = U_2$, where $U_1$ and $U_2$ are deterministic constants, the problem admits a unique solution $U \in L^2(\Omega; \textcolor{black}{C([0,T]))}$ satisfying~\cref{defn:soln-of-non-wick-sde}. Moreover, this solution is given by
              \begin{equation*}
                  U(t) = U_1 + \int_{0}^t \left( U_2 + F(s)  \right) e^{\diamond X_\beta(s)} \, ds, \qquad t \in (0,T).
              \end{equation*}
        \item\label{item:non-wick:dirichlet} \emph{Dirichlet data}: $U(0)=U_1, U(T)=U_2$, where $U_1$ and $U_2$ are deterministic constants, the problem admits a unique solution $U \in L^2(\Omega; \textcolor{black}{C([0,T]))}$ satisfying~\cref{defn:soln-of-non-wick-sde}. Moreover, this solution is given by
              \begin{equation*}
                  U(t) = U_1 + \int_{0}^t (\kappa + F(s)) e^{\diamond X_\beta(s)} \, ds, \qquad t \in (0,T),
              \end{equation*}
              where the random variable $\kappa$ is given by
              \begin{equation*}
                  \kappa = \left(U_2-U_1 - \int_{0}^T F(s) e^{\diamond X_\beta(s)} \, ds \right)\cdot \left(\int_{0}^T e^{\diamond X_\beta(s) }\, ds\right)^{-1},
              \end{equation*}
              with $\kappa$ almost surely bounded.
        \item\label{item:non-wick:neumann} \emph{Neumann boundary conditions}: $e^{\EX [X_\beta^2]} e^{\diamond (- X_\beta(0))} \cdot U'(0) = U_1$ and $e^{\EX [X_\beta^2]}e^{\diamond (-X_\beta(T))} \cdot U'(T) = U_2$, where $U_1$ and $U_2$ are deterministic constants, the problem admits a unique (up to additive random variable) solution $U \in L^2(\Omega; \textcolor{black}{C([0,T]))}$ satisfying~\cref{defn:soln-of-non-wick-sde}. Moreover, this solution is given by
              \begin{equation*}
                  U(t) = \int_{0}^t \left(U_1 + F(s) \right) e^{\diamond X_\beta(s)} \, ds, \qquad t \in (0,T).
              \end{equation*}
        \item\label{item:non-wick:periodic} \emph{Periodic boundary conditions}: $U(0) = U(T), e^{\diamond (-X_\beta(0))} U'(0)=e^{\diamond (-X_\beta(T))} U'(T)$, assuming that $\int_{0}^T f(s) \, ds=0$, the problem admits a unique (up to additive random variable) solution $U \in L^2(\Omega; \textcolor{black}{C([0,T]))}$ satisfying~\cref{defn:soln-of-non-wick-sde}. Moreover, this solution is given by
              \begin{align*}
                  U(t) =  \int_{0}^t (\kappa + F(s)) e^{\diamond X_\beta(s)} \,ds, \qquad t \in (0,T),
              \end{align*}
              and the random variable $\kappa$ is given by
              \begin{align*}
                  \kappa = -\left(\int_{0}^T F(s) e^{\diamond X_\beta(s) \, ds}\right) \cdot \left(\int_{0}^T e^{\diamond X_\beta(s)}\,ds\right)^{-1},
              \end{align*}
              with $\kappa$ almost surely bounded.
    \end{enumerate}
    We recall that $e^{\diamond X_\beta(s)} \, ds$ denotes the Gaussian multiplicative chaos measure.
\end{thm}

\begin{remark}
    We remark that although we have restricted ourselves to the so-called $L^2$-range of the Gaussian multiplicative chaos, i.e. $\beta^2 < 1$, the above~\cref{thm:non-wick-eqn-is-solvable} remains valid in the full range $\beta^2 < 2$. The key modification in the statement in this case is that the solution would belong only to  $L^p(\Omega; \textcolor{black}{C([0,T]))}$ for certain $p<2$. This is a consequence of the fact that for $\beta\geq 1$, the  Gaussian multiplicative chaos would have fewer moments, see~\cref{thm:GMC-moments} above.
\end{remark}

Before moving on to the proof of~\cref{thm:non-wick-eqn-is-solvable}, we make the following remarks on the sample path continuity of the solution, see also~\cref{rem:wick-soln-cont} below.

\begin{remark}\label{rem:non-wick-continuity}
    Note that from the representations of solutions, we observe that the solutions are in fact almost surely H\"older continuous of certain order, depending on the parameter $\beta$. This follows immediately from the solution formulas and the fact that the chaos measure $d\mu_\beta$ have almost sure modulus of continuity:
    \begin{equation*}
        \mu_\beta([s,t]) \leq C|t-s|^{\eta},
    \end{equation*}
    where $\eta= \eta(\beta)>0$ and $C$ is an almost surely finite random constant. For this fact, see~\cite[Proposition 2.5]{Bertacco2022} and the references therein.
\end{remark}

\begin{proof}[Proof of~\cref{thm:non-wick-eqn-is-solvable}]

    We prove the \textcolor{black}{t}heorem in the case of periodic boundary values, while the other cases can be proved with similar arguments. Let now $f \in L^1([0,T])$ be a mean-zero function and consider
    \begin{align*}
        -e^{\EX [X_\beta(t)^2]}(e^{\diamond (-X_\beta(t))} \cdot U'(t)  )' = f(t), \qquad t \in (0,T),
    \end{align*}
    with periodic boundary conditions
    \begin{align}
        U(0)                                   & =U(T) \nonumber \\
        e^{\diamond (-X_\beta(0))} \cdot U'(0) & =
        e^{\diamond (-X_\beta(T))} \cdot U'(T). \label{eq:periodic-cond-on-derivative}
    \end{align}

    \textit{Step 1. Approximation of the field $X$ with $X_\eps$.}

    We use~\cref{lem:mollification-of-log-correlated-field} to obtain a convolution approximation of the log-correlated field $X_\beta$, denoted by $X_{\beta, \eps}$, with $\eps > 0$, and we denote by $U_\eps$ the unknown function in the corresponding mollified differential equation.

    \textit{Step 2. Solving the regularized equation.}
    Now as the field $X_{\beta,\eps}$ is well-defined and smooth, directly integrating (recall that $\EX [X_{\beta,\eps}(t)^2]$ is independent of $t$) the equation we have, for fixed $\omega \in \Omega$,
    \begin{align*}
        e^{\diamond (-X_{\beta, \eps}(t))} U_\eps'(t) - e^{\diamond (-X_{\beta, \eps}(0))} U_\eps'(0) = e^{-\EX [X_{\beta, \eps}^2]} F(t), \qquad t \in (0,T)
    \end{align*}
    where we denote $F(t) := -\int_{0}^t f(s) \, ds$. We note that we see here that~\cref{eq:periodic-cond-on-derivative} is satisfied precisely when $f$ is mean-zero on $[0,T]$.

    ~\cref{def:wick-exponential} yields \textcolor{black}{for any centered Gaussian $Z$ that
    \begin{align*}
        e^{\diamond (-Z)} &= e^{-Z - \EX(-Z)^2 /2} = e^{-Z - \EX Z^2 /2 + \EX Z^2 /2 - \EX Z^2 /2 } \\ 
        &= \left( e^{Z - \EX Z^2 /2} \right)^{-1} \left( e^{\EX Z^2}  \right)^{-1}  = \left(e^{\diamond Z}  \right)^{-1}  \left( e^{\EX Z^2}  \right)^{-1}.
    \end{align*}
        }
    \textcolor{black}{Hence, for any} convolution approximation $X_{\beta,\eps}$ of $X_\beta$ it holds
    \begin{align*}
        e^{\diamond (-X_{\beta,\eps}(t))} = \frac{1}{e^{\EX [X_{\beta,\eps}(t)^2]} e^{\diamond X_{\beta,\eps}(t)}}
        , \qquad t \in (0,T).
    \end{align*}

    Using the above identity, after rearranging gives
    \begin{align*}
        U_\eps'(t) = \left( \frac{1}{e^{\diamond X_{\beta,\eps}(0)}} U_\eps'(0) + F(t) \right)e^{\diamond X_{\beta,\eps}(t)}, \qquad t \in (0,T).
    \end{align*}

    Integrating this expression up to $t \in (0,T)$, \textcolor{black}{we obtain
    \begin{align*}
        U_\eps(t) = U_\eps(0) + \int_0^t \left( \frac{1}{e^{\diamond X_{\beta,\eps}(0)}} U_\eps'(0) + F(s) \right)e^{\diamond X_{\beta,\eps}(s)} \, ds.
    \end{align*}
   The boundary condition $U_\eps(0) = U_\eps(T)$ now implies that
    \begin{align*}
         \int_0^T \frac{1}{e^{\diamond X_{\beta,\eps}(0)}} U_\eps'(0) e^{\diamond X_{\beta,\eps}(s)} \, ds + \int_0^T  F(s) e^{\diamond X_{\beta,\eps}(s)} \, ds= 0,
    \end{align*}
    and therefore we may write the solution as} 
    \begin{align*}
        U_\eps(t) = U_\eps(0) + \int_{0}^t \left(\kappa_\eps  + F(s) \right) e^{\diamond X_{\beta,\eps}(s)} \, ds,
    \end{align*}
    where the random constant $\kappa_\eps$ is 
    \begin{align}
        \kappa_\eps := \frac{1}{e^{\diamond X_{\beta,\eps}(0)}} U_\eps'(0) =- \frac{\int_{0}^T F(s) e^{\diamond X_{\beta,\eps}(s) }\,ds}{\int_{0}^T e^{\diamond X_{\beta,\eps}(s)}\, ds}.
    \end{align}

    \textit{Step 3. Convergence of the approximate solution.}
    In Step 2 we obtained the solution formula for the approximated equation as
    \begin{align}\label{eq:non-wick-approx-solution-periodic}
        U_\eps(t) = U_\eps(0) + \int_{0}^t \Big( F(s) + \kappa_\eps\Big) e^{\diamond X_{\beta, \eps}(s)}\, ds, \qquad t \in (0,T)
    \end{align}
    with $\kappa_\eps =- \frac{\int_{0}^T F(s) e^{\diamond X_{\beta,\eps}(s) }\,ds}{\int_{0}^T e^{\diamond X_{\beta,\eps}(s) }\, ds}$, for any $\eps > 0$.	Here
    \begin{equation*}
        |\kappa_\eps| \leq \sup_{s\in[0,T]}|F(s)| \leq \int_0^T |f(s)|ds  < \infty.
    \end{equation*}
    In order to complete the proof, we need the following result.
    \begin{lemma}\label{lemma:C-convergence}
        Assume that $\mu_n, \mu$ ($n\geq1$) are random positive measures on $[0,1]$ such that
        $\mu_n\to\mu$ in probability $($with respect to the weak star convergence of measures on $[0,1])$. Assume that $\mu$ is almost surely non-atomic. Let $h\in C([0,1])$ and write
        $f(x):=\int_0^xh(y)d\mu(y)$, and $f_n(x):=\int_0^xh(y)d\mu_n(y)$. Then $f_n\to f$ in probability in the space $C([0,1])$.
    \end{lemma}
    \begin{proof}[Proof of Lemma \ref{lemma:C-convergence}]
        Denote $A:=\|h\|_\infty$. Fix an integer $N\geq 2$ and choose a continuous partition of unity of tent-functions $\{\psi_k\,: \;,k=0,\ldots, 2N\}$ by setting $\psi_k(x)=\max (0, 1-2N |x-k(2N)^{-1}|)$, $x\in [0,1]$. Let  $\varepsilon=\varepsilon(\omega):=\sup_{|x-y|\leq 1/N} \mu([x,y])$. Set also
        \begin{equation*}
            A_{k,N} = \left\{\omega: \bigg| \int_0^1 \psi_k(y)d\mu_n(y)- \int_0^1 \psi_k(y)d\mu(y)\bigg|\leq  N^{-2}\right\}
        \end{equation*}
        and
        \begin{equation*}
            B_{k,N} =\left\{\omega: \bigg| \int_0^1 h(y)\psi_k(y)d\mu_n(y)- \int_0^1 h(y)\psi_k(y)d\mu(y)\bigg|\leq  N^{-2}\right\}.
        \end{equation*}
        By the assumption we may choose $n_0$ so large that for $n\geq n_0$ it holds,
        for any $k\in\{0,\ldots 2N\}$ and $n\geq n_0 $, that
        \begin{equation}
            \label{eq:tapahtuma}
            \PR\big(A_{k,N}\cap B_{k,N} \big) \geq 1 - \frac{1}{N(2N+1)}.
        \end{equation}
        Assume that $n\geq n_0$ and that the events $A_{k,N},B_{k,N}$ happen for all $k \in \{0,\ldots,2N\}$, which occurs at least with probability $1-1/N$, and let $x,y\in [0,1]$ be arbitrary such that $|x-y|\leq 1/N$. \textcolor{black}{As we have assumed that the event $A_{k,N}$ occurs for all $k \leq 2N$, and since} the measures $\mu_n, \mu$ are positive\textcolor{black}{,} we may majorize $\chi_{[x,y]}$ by using at most three of the functions
        $\psi_k$, \textcolor{black}{and thus} it follows that
        \begin{equation*}
            \mu_n([x,y])\leq 3(\varepsilon+1/N^2).
        \end{equation*}
        More generally,
        for any $x,y\in [0,1]$ we may express  $\chi_{[x,y]}$ as a sum of at most $2N+1$ of the functions $\psi_k$ and, in addition, two halves of a single tent function $\psi_k$. Then clearly
        \begin{align*}
            |f_n(x)-f(x)|&\leq (2N+1)N^{-2}+ 2A\big(\varepsilon + 3(\varepsilon+1/N^2)\big) < 8A\varepsilon+9\max(1,A)/N\\
            &< 9(A+1)(\varepsilon+1/N).
        \end{align*}
        Let next $\delta>0$ be given. \textcolor{black}{Since $\mu$ is a non-atomic measure, we have $\varepsilon \to 0$ in probability as $N\to \infty$. Thus we can} fix $N$ so large that $\PR\big(9(A+1)\varepsilon <\delta/2\big)\geq 1-\delta/2$ and $9(A+1)/N<\delta/2$. Then it follows for $n\geq n_0$  that
        \begin{equation*}
            \PR\big(\|f-f_n\|_{C([0,1])}<\delta\big)\geq 1-\delta/2-1/N>1-\delta.
        \end{equation*}
        This completes the proof of Lemma \ref{lemma:C-convergence}.
    \end{proof}
    We now continue the proof of \textcolor{black}{T}heorem \textcolor{black}{\ref{thm:non-wick-eqn-is-solvable}.} \textcolor{black}{It remains to show that \eqref{eq:non-wick-approx-solution-periodic} defines a solution in the sense of Definition \ref{defn:soln-of-non-wick-sde}.}
    By the weak \textcolor{black}{star} convergence in probability of $d\mu_{\beta,\eps}$ to $d\mu_\beta$, and the fact that $\mu_\beta$ is non-atomic almost surely, $\kappa_\eps \to \kappa$ in probability.
    Since $h$ is continuous, the convergence in probability in the space $C([0,T])$ now follows from~\cref{lemma:C-convergence}. \textcolor{black}{This} shows that we have obtained \textcolor{black}{a solution to problem \eqref{eq:non-wick-equation-general-problem}, and this solution is unique; see Remark \ref{rem:solution-uniqueness} above.} Finally, the fact that $U \in L^2(\Omega;\textcolor{black}{C([0,T]))}$ in cases~\cref{item:non-wick:initial,item:non-wick:neumann} follows directly from the representation formula together with the fact that $h$ is bounded and
    \begin{equation*}
        \sup_{0\leq t \leq T}\mu_\beta([0,t]) = \mu_\beta([0,T]) \in L^2(\Omega),
    \end{equation*}
    since $\beta<1$.
    Similarly, $U \in L^2(\Omega;\textcolor{black}{C([0,T]))}$ in case~\cref{item:non-wick:periodic}, since then $|\kappa| \leq \sup_{0\leq s\leq T}|F(s)|$ almost surely. Then, in case~\cref{item:non-wick:dirichlet} we have
    \begin{equation*}
        \frac{U_2-U_1}{\mu_\beta([0,T])}\int_0^t |F(s)|e^{\diamond X_\beta(s)}ds \leq (U_2-U_1)\sup_{0\leq s\leq T}|F(s)|
    \end{equation*}
    from which similar arguments as above yields $U \in L^2(\Omega;\textcolor{black}{C([0,T]))}$. \textcolor{black}{The uniqueness (or uniqueness up to a random variable as in (iii) and (iv)) of the solution follows from the facts that the approximated equation admits a unique solution, and limits are unique, see Remark \ref{rem:solution-uniqueness}.} This completes the whole proof.
\end{proof}

\section{Unique solvability of Wick equation}\label{sec:wick}
In this section we consider the problem
\begin{align}\label{eq:wick-equation-general-problem}
    \begin{cases}
        -(e^{\diamond (-X_\beta(t))} \diamond U'(t))' & = f(t), \qquad t \in (0,T) \\
        \textnormal{boundary data for } U.
    \end{cases}
\end{align}
We state and prove an existence and uniqueness result for integrable, deterministic functions $f$. Above the symbol $\diamond$ denotes the Wick product, see~\cref{def:wick-prod} below, and $e^{\diamond X_\beta}$ denotes the Wick exponential of $X_\beta$, as in~\cref{def:wick-exponential}. The boundary data can be taken to be initial values, Dirichlet, Neumann or periodic.
Again, as point-wise evaluations do not make sense for the log-correlated field $X_\beta$, we define the solution to the problem in~\cref{eq:wick-equation-general-problem} via approximations of the field using~\cref{lem:mollification-of-log-correlated-field}. Then, we convert this to a family of deterministic equations by taking an $S$-transform (see the definition below). These deterministic equations have unique solutions, from this we get a solution to~\cref{eq:wick-equation-general-problem} by using the invertibility of the $S$-transform.

We first recall some elementary facts about Gaussian $L^2$-spaces, and then define the $S$-transform, one of the main tools of the remainder of the paper.

We recall that a Gaussian linear space $\mathcal{H}$ is a vector space of real-valued random variables, defined on a probability space $(\Omega, \mathcal{F},\PR)$, such that each random variable follows a Gaussian distribution with mean zero. Each $h \in \mathcal{H}$ has finite variance; therefore $\mathcal{H} \subset L^2(\Omega, \mathcal{F},\PR)$. If $\mathcal{H}$ is a closed subspace of $L^2(\Omega, \mathcal{F},\PR)$, we say that $\mathcal{H}$ is a Gaussian Hilbert space. It is well known that if $\mathcal{G} \subset L^2(\Omega, \mathcal{F},\PR)$ is a Gaussian linear space, then the closure of $\mathcal{G}$ is a Gaussian Hilbert space.

\begin{definition}[Gaussian Hilbert space generated by Gaussian families]
    Let $\mathcal{H}$ be a Gaussian Hilbert space. For any collection $\{\xi_\alpha\}$ of centered jointly normal random variables in $\mathcal{H}$, the closed linear span of $\{\xi_\alpha\}$ in $L^2(\Omega,\mathcal{F},\PR)$ is called the Gaussian Hilbert space generated by $\{\xi_\alpha\}$.
\end{definition}

We also recall the basic Wiener-It\^{o} chaos decomposition of $L^2$: For $n \geq 0$, let $\textcolor{black}{\overline{\mathcal{P}}}_n(\mathcal{H})$ be the closure in $L^2(\Omega,\mathcal{F},\PR)$ of the linear space
\begin{equation*}
    \mathcal{P}_n(\mathcal{H}) = \{ p(\xi_1, \ldots,\xi_m): p \textnormal{ is a polynomial of degree } \leq n, \xi_1, \ldots, \xi_m \in \mathcal{H}, m < \infty\},
\end{equation*}
and set $K^n := \overline{\mathcal{P}}_n(\mathcal{H}) \cap \overline{\mathcal{P}}_{n-1}(\mathcal{H})^\perp$, with $K^0 = \overline{\mathcal{P}}_0(\mathcal{H})$ \textcolor{black}{being the space of constants}. The space $K^n$ is called the \textit{n\textsuperscript{th} (homogeneous) chaos}. It is clear from the construction that $K^n \subset L^p(\Omega,\mathcal{F},\PR)$ for any $p\geq 1$. We also remark that, by construction the first chaos $K_1$ is composed of Gaussian random variables. More importantly, we have (see, for instance~\cite[Theorem 2.6]{janson1997}) the orthogonal decomposition
\begin{equation}\label{eq:wiener-ito-decomposition}
    \bigoplus_{n=0}^\infty K^n = L^2\left(\Omega,\mathcal{F}_{\mathcal{H}},\PR\right) =: L^2_{\mathcal{H}}(\Omega),
\end{equation}
where $\mathcal{F}_{\mathcal{H}}$ is the $\sigma$-algebra generated by the Gaussian random variables in $\mathcal{H}$. Actually, the space $L^2_{\mathcal{H}}(\Omega)$ is spanned by the Wick exponential of Gaussian random variables $h\in \mathcal{H}$, see~\cite[Corollary 3.40]{janson1997} and~\cref{def:wick-exponential}. The space $L^p_\mathcal{H}(\Omega)$ for $p\geq 1$ is defined as $\mathcal{F}_{\mathcal{H}}$-measurable random variables that have finite $p$ moment. Also, note that $e^{\diamond h} \in L^p_\mathcal{H}(\Omega)$ \textcolor{black}{for all $h \in \mathcal{H}$} and for all $p\geq 1$. Furthermore, one may prove that $K^n \subset L^p_{\mathcal{H}}(\Omega)$ for all $p,n\geq 1$.
\begin{definition}[Wick product]\label{def:wick-prod}
    Let $Z \in K^m, Y \in K^n$. We define the \emph{Wick product} of $Z$ and $Y$ by
    \begin{equation*}
        Z \diamond Y  = \pi_{m+n}(ZY),
    \end{equation*}
    where $\pi_n$ denotes the orthogonal projection of $L^2_\mathcal{H}(\Omega)$ onto $K^n$.
    The operator $\diamond$ can by bilinearity be extended to a \textcolor{black}{bilinear} operator on the space of all elements of $L_\mathcal{H}^2(\Omega)$ having finite chaos decomposition.

    We define the Wick powers of $Z^{\diamond n}, n=0,1,2,\ldots$ inductively as
    \begin{equation*}
        \begin{cases} Z^{\diamond 0} & = 1,                                                \\
              Z^{\diamond n} & = Z \diamond Z^{\diamond(n-1)}, \qquad n=1,2,\ldots
        \end{cases}
    \end{equation*}
\end{definition}

\begin{definition}[$S$-transform]
    Let $\mathcal{H}$ be a Gaussian Hilbert space, and let $Z\in L_\mathcal{H}^2(\Omega)$.
    We call the mapping $S\textcolor{black}{Z} \colon \mathcal{H} \to \mathbb{R}$ given by
    \begin{equation}
        \label{eq:S-transform-basic}
        (SZ)(h) = \EX \left[Ze^{\diamond h}\right], \quad h \in \mathcal{H},
    \end{equation}
    the $S$-transform of $Z$ evaluated at $h$. Actually, it suffices to assume $Z \in L^p_\mathcal{H}(\Omega)$ for some $p>1$.
\end{definition}
\begin{remark}
    We remark that by the density of \textcolor{black}{\{$e^{\diamond h}, h$\}} in $L_\mathcal{H}^2(\Omega)$, it is a routine task to show that the $S$-transform is a linear, injective map on $L_\mathcal{H}^2(\Omega)$: that is, if $(SZ_1)(h)=(SZ_2)(h)$ for all $h \in \mathcal{H}$, then $Z_1 = Z_2$ almost surely. In fact, it is sufficient to have $(SZ_1)(h) = (SZ_2)(h)$ for all $h\in \mathcal{A}$, where $\mathcal{A}$ is a dense subset of $\mathcal{H}$. For details, see~\cite[Theorem 16.11, Corollary 3.40]{janson1997}.
\end{remark}
The Wick product acts like an ordinary product under $S$-transform.
\begin{lemma}\label{lem:S-transform-wick-product-scalar-product}
    Let $\mathcal{H}$ be a Gaussian Hilbert space, and $Y \in K^m, Z \in K^n$. Then we have
    \begin{equation*}
        S\left(Y \diamond Z   \right)(h) = S(Y)(h) \cdot S(Z)(h), \qquad \textnormal{for all } h \in \mathcal{H}.
    \end{equation*}
\end{lemma}
\textcolor{black}{For a proof, see \cite{janson1997}.}

The content of~\cref{lem:S-transform-wick-product-scalar-product} can be used as a defining property of a general Wick product.
\begin{definition}[General Wick-product]\label{def:general-wick-product}
    Let $Y,Z \in \bigcup_{q>1} L^q_\mathcal{H}(\Omega)$, and $p>1$. We say that the Wick product of $Z$ and $Y$ exists in $L_\mathcal{H}^p(\Omega)$ if there exists an element $V \in L_\mathcal{H}^p(\Omega)$ such that
    \begin{equation*}
        (SY)(h) \cdot (SZ)(h) = (SV)(h),
    \end{equation*}
    for all $h \in \mathcal{H}$, and we denote $V$ by $Y \diamond Z$.
\end{definition}
It is in general not true that for $Y,Z \in L^p_{\mathcal{H}}(\Omega)$, the Wick product $Y \diamond Z$ is still an element of $L^p_{\mathcal{H}}(\Omega)$ or even $L^1_{\mathcal{H}}(\Omega)$ (or that it even exists), see e.g.~\cite{gjessing1991} or~\cite[Example 16.33]{janson1997}. To overcome this, one can enlarge the space and consider random variables $Y,Z \in (\mathcal{S})_{-1}$, where $(\mathcal{S})_{-1}$ is a stochastic distribution space (Hida space), \textcolor{black}{which we briefly outline next. We will list here only elementary definitions and properties without proofs; for a more thorough treatment, see for instance~~\cite{gjessing1991},~\cite{hu2009}, \cite{holden2009} or~\cite{kondratiev1996} and the references therein.}

\textcolor{black}{
We start by recalling that via the Bochner-Minlos theorem we obtain a probability measure $\gamma$ on $\mathcal{B}(\mathscr{S}'(\R))$, called the one-dimensional white noise probability measure, such that
\begin{align*}
    \EX_\gamma [e^{i\langle \cdot, \varphi\rangle}] = e^{-\frac{1}{2} \norm{\varphi}^2}, \qquad \forall \varphi \in \mathscr{S}(\R).
\end{align*}
Above $\mathscr{S}$ denotes the Schwartz space of rapidly decreasing functions on $\R$, and $\mathscr{S}'$ its topological dual. \textcolor{black}{To construct an explicit basis for $L^2(\gamma)$, one starts with the classical Hermite functions $(\xi_j)_{j=0}^\infty$, which form an orthonormal basis for $L^2(\R)$. A corresponding orthonormal basis for $L^2(\R^d)$ is then given by the tensor products of these functions. Indeed, for each multi-index $k=(k_1, \dots, k_d) \in \mathbb{N}_0^d$, we define
\begin{align*}
    \eta_k(x_1, \dots, x_d) := \xi_{k_1}(x_1) \otimes \dots \otimes \xi_{k_d}(x_d).
\end{align*}
Since the set of multi-indices $\mathbb{N}_0^d$ is countable, we can arrange the basis functions $(\eta_k)_{k \in \mathbb{N}_0^d}$ into a single sequence, which we denote by $(\eta_i)_{i=1}^\infty$.} Moreover, we recall the classical orthogonal Hermite polynomials defined by
\begin{align}\label{eq:hermite-polynomial-defn}
    h_n(x)  = (-1)^n e^{\frac{1}{2}x^2} \frac{d^n}{dx^n} \left(e^{-\frac{1}{2}x^2}   \right), \qquad n=0,1,2,\hdots
\end{align} } 

\textcolor{black}{
Denoting by $\mathcal{J}$ the space of all sequences $\alpha = (\alpha_1, \alpha_2, \hdots)$ with all coordinates $\alpha_j \in \N$ such that $\alpha_j \neq 0$ for only finitely many indices $j$, and setting 
\begin{align}\label{eq:white-noise-basis-1D}
        H_{\alpha}(\omega) := \prod_{i=1}^\infty h_{\alpha_i}(\langle \omega, \eta_i \rangle), \qquad \omega \in \mathscr{S}'(\R)
    \end{align}
yields a basis for $L^2(\gamma)$ (i.e. for \eqref{eq:wiener-ito-decomposition}). Here $h_{\alpha_i}$ is the Hermite polynomial of degree $\alpha_i$ given by~\cref{eq:hermite-polynomial-defn}. 
\textcolor{black}{In the sequel, if $\gamma = (\gamma_1, \dots, \gamma_j, \dots)$, with $\gamma_j \neq 0$ for only finitely many $j$, we write
\begin{align*}
    (2 \N )^\gamma := \prod_j (2j)^{\gamma_j}.
\end{align*}
}
By introducing different summability criteria for the coefficients \textcolor{black}{of the formal series expansion} (than \textcolor{black}{just} the plain $\ell^2$ one) we obtain stochastic test function and distribution spaces as follows:} 
    \begin{definition}[Stochastic test function spaces] \label{def:stochastic-test-function-space}
    For $\rho \in [\textcolor{black}{0},1]$ and $r \in \R_{\textcolor{black}{+}}$,  let $(\mathcal{S})_{\rho,r}$ consist of those $F = \sum_{\alpha \in \mathcal{J}} a_\alpha H_\alpha 
        \in \mathbf{L}^2(\gamma)$
    such that
    \begin{align*}
        \scnorm{F}_{\rho,r}^2 := \sum_{\alpha \in \mathcal{J}} a_\alpha^2 (\alpha!)^{1+\rho} (2\N)^{r\alpha} < \infty.
    \end{align*}
\end{definition}

\begin{definition}[Stochastic distribution spaces]\label{def:stochastic-distribution-space}
    For $0 \leq \rho \leq 1$ and $q \in \R_{\textcolor{black}{+}}$, let $(\mathcal{S})_{-\rho, -q}$ consist of all formal expansions
    \begin{align*}
        F = \sum_{\alpha} b_\alpha H_\alpha, \qquad b_\alpha \in \R
    \end{align*}
    such that
    \begin{align*}
        \scnorm{F}_{-\rho,-q}^2 := \sum_{\alpha} b_\alpha^2(\alpha!)^{1-\rho} (2\N)^{-q\alpha} < \infty.
    \end{align*}
\end{definition}

\textcolor{black}{We remark that above the parameter $\rho$ essentially measures the 'stochastic smoothness' of $F$: for instance, for $\rho > 0$, the higher-order chaos coefficients $a_\alpha$ need to decay quite rapidly, due to the factorial of $\alpha$. Furthermore, the parameter $r$ can be viewed as acting as a measure of 'analytic smoothness', as in terms of oscillation, given that it weighs the basis 'frequency' with a factor of $2\N$.}

\textcolor{black}{We define for $\rho \in [0,1]$ the space $(\mathcal{S})_{\rho}$ as the projective limit (intersection) of the spaces $(\mathcal{S})_{\rho, r}$ as in~\cref{def:stochastic-test-function-space} with respect to $r \in \R_{\textcolor{black}{+}}$, and similarly $(\mathcal{S})_{-\rho}$ is the inductive limit (union) of the spaces $(\mathcal{S})_{-\rho, \textcolor{black}{-q}}$ as in Definition \ref{def:stochastic-distribution-space}.}

\textcolor{black}{We note that akin to the case of deterministic distribution and test function spaces, we have, in general for $0 \leq \rho \leq 1$ the inclusions 
\begin{align*}
    (\mathcal{S})_{1} \subset (\mathcal{S})_{\rho} \subset (\mathcal{S})_{0} \subset L^2(\gamma) \subset (\mathcal{S})_{-0} \subset (\mathcal{S})_{-\rho} \subset (\mathcal{S})_{-1},
\end{align*}
where each embedding is continuous and dense.}

\textcolor{black}{
\begin{definition}[Wick product]
    For $F$ and $G$, two elements of $(\mathcal{S})_{-1}$, that is,
    \begin{align*}
        F = \sum_{\alpha} a_\alpha H_{\alpha}(\omega), \quad G = \sum_{\beta} b_\beta H_\beta(\omega),
    \end{align*}
    with $a_\alpha, b_\beta \in \R$, the Wick product $F \diamond G$ is defined by
    \begin{align*}
        F \diamond G = \sum_{\alpha, \beta}  a_\alpha  b_\beta H_{\alpha + \beta} (\omega).
    \end{align*}
\end{definition} 
One key advantage of the spaces $(\mathcal{S})_1$ and $(\mathcal{S})_{-1}$ is that they are closed with respect to the Wick product: if $F, G \in (\mathcal{S})_1$, then $F \diamond G \in (\mathcal{S})_1$; similarly if $F, G \in (\mathcal{S})_{-1} $, then $F \diamond G \in (\mathcal{S})_{-1}$. In addition, in the stochastic distribution space $(\mathcal{S})_{-1}$ the Wick product is commutative, distributive and associative.}

To complement the algebraic properties of the Wick product, for any natural number $k$, we define the Wick powers $X^{\diamond k}$ of $X \in (\mathcal{S})_{-1}$ inductively by setting $X^{\diamond 0}=1$ and then
\begin{align*}
    X^{\diamond  k} = X \diamond X^{\diamond (k-1)}, \qquad k > 1.
\end{align*}
We remark that this allows for a generalization to $(\mathcal{S})_{-1}$ of the Wick exponential identity in \eqref{eq:wick-exponential-series} as in Remark \ref{rem:wick-exponential} above. See also \eqref{eq:wick-property4} below for the Wick inverse. \textcolor{black}{In the sequel, we will often write $w(\varphi)$ for the dual pairing of $\langle w, \varphi \rangle$.}

\begin{definition}[Hida test points]
    \label{def:Hida-test}
    Let $\varphi\in \mathscr{S}(\R)$ and $\lambda \in \mathbb{R}$. For given $\lambda$,
    we call the random variables $\lambda w( \varphi)$ the Hida test points.
\end{definition}
The name in~\cref{def:Hida-test} stems from the fact that if $|\lambda|$ is small enough, then
$e^{\diamond \lambda w( \varphi)} \in ((\mathcal{S})_{-1,-q})^\ast = (\mathcal{S})_{1,q}$ and they are dense in $(\mathcal{S})_{1,q}$. Consequently, we can define the $S$-transform that determines random variables $F \in (\mathcal{S})_{-1,-q}$ uniquely:

\begin{definition}[S-transform]\leavevmode
    \label{def:S-transform-Hida}
    \textcolor{black}{\begin{enumerate}
        \item Let $F \in (\mathcal{S})_{-1}$ and let $\varphi \in \mathscr{S}(\R)$. Then the $S$-transform of $F$ at $\lambda w(\varphi)$, is defined for all real numbers $\lambda$ with $\abs{\lambda}$ small enough, by
            \begin{align*}
                (SF)(\lambda w(\varphi)) = \langle F, e^{\diamond \lambda w(\varphi)}\rangle ,
            \end{align*}
            where $\langle \cdot, \cdot \rangle$ denotes the action of $F \in (\mathcal{S})_{-1,-q}$ on $e^{\diamond \lambda w( \varphi)} \in ((\mathcal{S})_{-1,-q})^\ast = (\mathcal{S})_{1,q}$.
        \item Let $F \in (\mathcal{S})_{-\rho}$ for some $\rho < 1$ and let $\varphi \in \mathscr{S}(\R)$. Then the $S$-transform of $F$ at $\lambda w(\varphi)$ is defined by
            \begin{align*}
                (SF)(\lambda w(\phi)) = \langle F, e^{\diamond \lambda w(\varphi)} \rangle
            \end{align*}
            for all $\lambda \in \R$.
    \end{enumerate}}

\end{definition}

From the above definitions follow the expected identity\footnote{Actually, one could equivalently define the $S$-transform first and then define the Wick product through this identity, see e.g.~\cite{kondratiev1996}.}(cf. Lemma \ref{lem:S-transform-wick-product-scalar-product} above)
\begin{equation*}
    (S(Y\diamond Z))(h) = (SY)(h)\cdot (SZ)(h)
\end{equation*}
for suitable test objects $h \in \mathcal{H}$. We note however, that in this case the $S$-transform can be defined only for Gaussian objects $h$ that are small enough in suitable associated norms. More precisely, the choice of the associated norm depends on the subspace $(\mathcal{S})_{-1,-q} \subset (\mathcal{S})_{-1}$ one wishes to characterize by $S$-transform, where we have
\begin{equation*}
    (\mathcal{S})_{-1} = \bigcup_{q} (\mathcal{S})_{-1,-q}.
\end{equation*}
For $(\mathcal{S})_{-1,-q}$, the test points $h$ are Gaussian random variables such that $e^{\diamond h} \in (\mathcal{S})_{1,q}$ (see~\cref{def:S-transform-Hida}). In particular, this holds true if $\EX [h^2]$ is small enough, where the bound depends only on $q$.  In practice, this can be achieved by a multiplication with a small enough constant $\lambda$ and testing on Gaussian points $\lambda h$. Regardless, the $S$-transform (now defined on a restricted class of Gaussian random variables) characterizes random variables in $(\mathcal{S})_{-1,-q}$ uniquely. As a particular example, we note that whenever $Y$ and $Z$ are elements of spaces $L^p_{\mathcal{H}}(\Omega) \subset (\mathcal{S})_{-1}$, the generalized Wick product and $S$-transform equals to the ones given in~\cref{eq:S-transform-basic} and~\cref{def:general-wick-product}. In this case, the $S$-transform can be defined for any test point $h \in \mathcal{H}$ and the $S$-transform characterizes random variables uniquely. In the sequel, for the sake of notational simplicity we always use $h\in \mathcal{H}$ as a test point when taking $S$-transforms. The reader should keep in mind however that one should consider $\lambda h \in \mathcal{H}$ for small enough $\lambda$ if necessary, despite the fact that we have suppressed $\lambda$ from the notation. In what follows, we naturally assume that our log-correlated field is constructed from the white noise, allowing to use white noise analysis techniques. \textcolor{black}{On $\R$, one can set $X = (\Delta)^{-1/4} W$, where $(\Delta)^{-1/4}$ is the fractional Laplacian and $W$ is the white noise. For more details on this type of fractional Gaussian field construction also in more general domains, we refer the reader to the excellent survey \cite{lodhia2016}.}

For the reader's convenience we record some properties of Wick calculus in the next lemma. That is, the Wick product is commutative, associative and distributive, and it reduces to the usual scalar product when one of the factors is deterministic. On top of that, we collect some elementary properties of the Wick exponential to the same lemma. The proofs of the identities below (in the case of $L^p_\mathcal{H}(\Omega)$ random variables) can be found for instance in~\cite{janson1997}.

\begin{lemma}\label{lem:wick-identities}
    Let $X, Y, Z \in (\mathcal{S})_{-1}$. Then the Wick product is commutative, associative, and distributive, i.e.
    \begin{align}
        X \diamond Y              & = Y \diamond X, \nonumber                             \\
        X \diamond (Y \diamond Z) & = (X\diamond Y) \diamond Z, \label{eq:wick-property2} \\
        X \diamond (Y + Z)        & =X\diamond Y + X \diamond Z. \nonumber
    \end{align}
    Moreover, for any deterministic $C$ we have
    \begin{equation}
        X \diamond C = X \cdot C \label{eq:wick-property3}.
    \end{equation}
    Finally, for Gaussian random variables $X$ and $Y$ the Wick exponential satisfies
    \begin{align}
        e^{\diamond(X+Y)}    & =e^{\diamond X} \diamond e^{\diamond Y}  \label{eq:wick-exponential-sum-wick-product}  \\
        e^{\diamond( X + Y)} & = e^{-\EX [XY]} e^{\diamond X} e^{\diamond Y}, \label{eq:wick-exponential-sum-identity}
    \end{align}
\end{lemma}

\begin{remark} \label{rem:wick-exp}
    In particular, taking $Y=-X$ in~\cref{eq:wick-exponential-sum-identity} immediately results in
    \begin{align*}
        e^{\diamond (-X)} & =\frac{1}{e^{\EX [X^2]}e^{\diamond X}},
    \end{align*}
    upon recalling that by definition $e^{\diamond 0} = 1$. The above identity justifies the inclusion of the $e^{\diamond(-X)}$-term in problem~\cref{eq:wick-equation-general-problem} above, in addition to the fact that the log-correlated Gaussian field $X$ is symmetric: thus $e^{\diamond (-X)}$ and $e^{\diamond X}$ obey the same law.
\end{remark}
Random variables \textcolor{black}{$X = \sum_{\alpha}a_\alpha H_\alpha(\omega) \in (\mathcal{S})_{-1}$ such that $a_0 =: \EX[X] \neq 0$ (i.e. $X$ has non-zero (generalized) expectation)} also admit a multiplicative Wick inverse $X^{\diamond (-1)} \in (\mathcal{S})_{-1}$ defined through identity
\begin{align}
    X \diamond X^{\diamond (-1)} & = 1. \label{eq:wick-property4}
\end{align}
\begin{remark}\label{rem:wick-inverse-s-transform}
    In terms of $S$-transforms, we obtain
    \begin{align}
        \label{eq:inverse-S}
        (SX)(h) \cdot (SX^{\diamond(-1)})(h) = 1 \quad \textnormal{if and only if} \quad (SX^{\diamond(-1)})(h) = ((SX)(h))^{-1},
    \end{align}
    for the Wick inverse $X^{\diamond(-1)}$. In particular, we have
    \begin{align*}
        \frac{(SX)(h)}{(SY)(h)} & = (SX)(h) \cdot((SY)(h))^{-1} = (SX)(h) \cdot (SY^{\diamond(-1)})(h) \\
                                & =
        (S(X \diamond Y^{\diamond(-1)}))(h).
    \end{align*}
    We will return to this point in the proof of~\cref{thm:wick-eqn-is-solvable}.
\end{remark}

\begin{remark}\label{rem:wick-inverse}
    The Wick inverse in~\cref{eq:wick-property4} can be a very poorly behaving object in general. Indeed, even in the simple case $X(t)=1+\alpha Z$ with $Z\sim N(0,1)$ and $\alpha \in \mathbb{R}$, one can show by using the Wiener-It\^{o} chaos decomposition that, formally, we have
    \begin{equation*}
        X(t)^{\diamond (-1)} = 1+\sum_{k=1}^\infty (-\alpha)^k H_k(Z),
    \end{equation*}
    where $H_k$ denotes the normalized Hermite polynomial. Note that the above series converges in $L^2_{\mathcal{H}}(\Omega)$ only for $|\alpha|<1$. This example already illustrates that even for Gaussian random variables $X$ the inverse $X^{\diamond (-1)}$ in general is not a well-defined random variable (in some $L^p_\mathcal{H}(\Omega)$ space), but rather only a generalized function in the white noise space $(\mathcal{S})_{-1}$. For the computation of the Wick inverse in the white noise case, we refer to~\cite{hu1994}.
\end{remark}

\textcolor{black}{In general, due to the presence of this Wick inverse, our solutions of the Wick problem can be very rough or singular. Therefore, in contrast to  Definition \ref{defn:soln-of-non-wick-sde}, where we defined the solutions of the non-Wick problem essentially via the supremum norm, here we cannot expect that our solutions are always going to be continuous stochastic processes, or even belong to $L^2(\Omega)$. Therefore, we will introduce the following definition, which is a form of weak convergence, and the solution is a stochastic generalized function in general.}

\begin{definition}\label{def:wick-solution}
    Let $X_\eps$ be a convolution approximation of the log-correlated field $X$ on $[0,T]$ so that $d\mu_{\beta,\eps} \to d\mu_\beta$ in probability in the space of Radon measures (see~\cref{lem:mollification-of-log-correlated-field}) as $\eps \to 0$. Let $U_\eps$ be the corresponding solution of the problem
    \begin{align*}
        \begin{cases}
            -(e^{\diamond (-X_{\eps}(t))} \diamond U_\eps'(t))' & = f(t), \qquad t \in (0,T) \\
            \textnormal{boundary data for } U_\varepsilon.
        \end{cases}
    \end{align*}
    In the above, boundary data can be any of the cases~\cref{item:wick:initial,item:wick:dirichlet,item:non-wick:neumann,item:non-wick:periodic} in~\cref{thm:wick-eqn-is-solvable}.
    We say that $U$ with \textcolor{black}{$U(t) \in (\mathcal{S})_{-1,-q}$ for all $t\in(0,T)$ and for some $q$, is a solution of~\cref{eq:wick-equation-general-problem}, if $U_\eps(t) \in (\mathcal{S})_{-1,-q}$ for all $t\in(0,T)$ and (small enough) $\eps$ and 
    we have}
    \begin{equation*}
        \lim_{\eps \to 0}(SU_\eps(t))(h)=(SU(t))(h)
    \end{equation*} 
    for all $t \in (0,T)$ and for all Hida test points $h$ \textcolor{black}{as per Definition \ref{def:Hida-test}}.
\end{definition}
\textcolor{black}{Note that the definition essentially says that we seek solution from the space $(\mathcal{S})_{-1}$ (that is the union of $(\mathcal{S})_{-1,-q}$ for all $q$). The reason for specifying $q$ in the definition is that $S$-transform is defined via duality for test points $h \in (\mathcal{S})_{1,q}$ and hence we need to choose $q$ uniformly over $\epsilon$.}

We remark that although the convergence specified in~\cref{def:wick-solution} is even weaker than the weak $L^2(\Omega)$ convergence, with some choices of boundary data we can replace the convergence in the $S$-transform side with the usual weak $L^2$-convergence. In particular, this is the case with Neumann boundary data or in the initial value problem. 

\begin{thm} \label{thm:wick-eqn-is-solvable}
    Let $f \colon [0,T] \to \R$ be a deterministic, integrable function, and let $X_\beta=\beta X$ where $X$ is the log-correlated Gaussian field on $(0,T)$, with $0 < \beta < 1$. Consider the equation
    \begin{align*}
        -(e^{\diamond (-X_\beta(t))} \diamond U'(t))' = f(t), \qquad t \in (0,T).
    \end{align*}
    Denote $F(t) = -\int_0^t f(s) ds$, then \textcolor{black}{the above equation }with:
    \begin{enumerate}
        \item\label{item:wick:initial} \emph{Initial data}: $U(0)=U_1$ and $e^{\diamond (-X_\beta(0))}\diamond U'(0) = U_2$, where $U_1$ and $U_2$ are deterministic constants, the problem admits a unique solution $U$ on $(0,T)$ in the sense of~\cref{def:wick-solution}. The solution is given by
            \begin{align*}
                U(t) = U_1 + \int_{0}^t \left(U_2 + F(s)  \right) \diamond e^{\diamond X_\beta(s)} \, ds, \qquad t \in (0,T)
            \end{align*}

        \item\label{item:wick:dirichlet} \emph{Dirichlet data}: $U(0)=U_1$, and $U(T)=U_2$, where $U_1$ and $U_2$ are deterministic constants, the problem admits a unique solution $U$ on $(0,T)$ in the sense of~\cref{def:wick-solution}. The solution is given by
            \begin{align*} 
                U(t) = U_1 + \int_{0}^t (\kappa  + F(s)) \diamond e^{\diamond X_\beta(s)}\, ds, \qquad t \in (0,T)
            \end{align*}
            where the random variable $\kappa$ is 
            \begin{align*} 
                \kappa:=\left(U_2 - U_1 - \int_{0}^T F(s) e^{\diamond X_\beta(s)}\, ds \right) \diamond \left(\int_{0}^T e^{\diamond X_\beta(s)}\, ds  \right)^{\diamond(-1)}.
            \end{align*}

        \item\label{item:wick:neumann} \emph{Neumann boundary conditions}: $e^{\diamond (-X_\beta(0))} \diamond U'(0) =U_1$ and $e^{\diamond (-X_\beta(T))} \diamond U'(T) = U_2$, where $U_1$ and $U_2$ are deterministic constants, the problem admits a unique (up to an additive random variable) solution $U$ on $(0,T)$ in the sense of~\cref{def:wick-solution}. The solution is given by
            \begin{equation*}
                U(t) =  \int_{0}^t \left(U_1 + F(s) \right) \diamond e^{\diamond X_\beta(s)} \, ds, \qquad t \in (0,T).
            \end{equation*}
        \item\label{item:wick:periodic} \emph{Periodic boundary conditions}: $U(0)=U(T)$, and $e^{\diamond (-X_\beta(0))} \diamond U'(0) = e^{\diamond (-X_\beta(T))} \diamond U'(T)$, assuming that $\int_{0}^T f(s)\,ds = 0$, the problem admits a unique (up to an additive random variable) solution $U$ on $(0,T)$ in the sense of~\cref{def:wick-solution}. The solution is given by
            \begin{align} \label{eq:wick-solution-periodic}
                U(t) =   \int_{0}^t (\kappa + F(s))\diamond e^{\diamond X_\beta(s)} \,ds, \qquad t \in (0,T),
            \end{align}
            and the random variable $\kappa$ is 
            \begin{align} \label{eq:wick-solution-periodic-kappa}
                \kappa := -\left(\int_{0}^T F(s) e^{\diamond X_\beta(s) \, ds}\right)\diamond \left(\int_{0}^T e^{\diamond X_\beta(s)\,ds}\right)^{\diamond(-1)}.
            \end{align}
    \end{enumerate}
    Recall that $e^{\diamond X_\beta(s)}\, ds$ denotes the Gaussian multiplicative chaos measure.
\end{thm}

\begin{remark}
    We remark that cases~\cref{item:wick:initial,item:wick:neumann} above the solutions actually belong to $L^2(\Omega;C(0,T))$ and coincides with the solution to the renormalized non-Wick equation, see~\cref{thm:non-wick-eqn-is-solvable}. This follows from~\cref{thm:non-wick-eqn-is-solvable}, since in these cases
    \begin{equation*}
        \int_{0}^t \left(U_1 + F(s) \right) \diamond e^{\diamond X_\beta(s)} \, ds = \int_{0}^t \left(U_1 + F(s) \right) e^{\diamond X_\beta(s)} \, ds
    \end{equation*}
    by the fact that $U_1 + F(s)$ is deterministic.
\end{remark}

\begin{remark}\label{rem:wick-soln-cont}
    While the solution to the ``non-Wick'' problem
    \begin{equation*}
        e^{\EX [ X_\beta(t)^2]}(e^{\diamond (-X_\beta(t))}U'(t))' = f(t),
    \end{equation*}
    with given boundary data is a well-defined H\"older continuous process, see~\cref{rem:non-wick-continuity}, this is not the case for the Wick equation in general. Indeed, the reason being the (Wick) multiplication with a Wick inverse in the constant $\kappa$, cf.~\cref{rem:wick-inverse-s-transform}. The exceptions are the cases~\cref{item:wick:initial,item:wick:neumann}, where the Wick product reduces to the ordinary scalar product.
\end{remark}

\begin{remark}
    We remark that for $D \subset \R^n$ the pressure equation
    \begin{align*}
        \begin{cases}
            -\nabla (K(x) \diamond \nabla U(x)) = f(x), & \qquad x \in D    \\
            U(x) = 0,                                   & x \in \partial D,
        \end{cases}
    \end{align*}
    where $K$ is a re-normalized exponential of Gaussian white noise, has been studied in the literature. Common to some approaches in the literature and ours is the idea to transform the problem into a deterministic problem for which PDE techniques can be applied. For example, in~\cite{holden1995} the authors obtain a representation formula for the solution as a $(\mathcal{S})_{-1}$-valued random variable by using the Hermite transform, which is closely related to $S$-transform employed in the present paper. The key difference here is that for the white noise, upon Hermite (or S-) transformation the left-hand-side of the equation defines a uniformly elliptic operator and, thus is amenable to standard analytic methods. For a related result, see also~\cite{potthoff1992}. Finally, we note that in~\cite{lindstrom1991} (see also~\cite{holden2009}), the authors solve a one-dimensional variant of the above problem,
    \begin{align*}
        -(e^{\diamond X(t)}\diamond U'(t))'=f(t), \qquad t\in (0,T),
    \end{align*}
    directly by Wick calculus identities in~\cref{lem:wick-identities}.
    Indeed, integrating the equation, applying~\cref{eq:wick-property2},~\cref{eq:wick-property4} (since $\EX [e^{\diamond X}] =1$),~\cref{eq:wick-exponential-sum-identity} and~\cref{eq:wick-property3} we have
    \begin{align*}
        U'(t) = -\left(e^{\diamond X(t)}\right)^{\diamond(-1)} \int_{0}^t f(s) \, ds + C,
    \end{align*}
    where $C$ is an arbitrary random constant. Upon integrating once more, the above yields a formula for the solution in the distribution space $(\mathcal{S})_{-1}$.
\end{remark}

Before proceeding with the proof of~\cref{thm:wick-eqn-is-solvable} we need the following technical lemma to handle the random variable $\kappa$.

\begin{lemma}\label{lemma:S_q-convergence}
    Let $X_\eps \to X$ in $L^2_\mathcal{H}(\Omega)$, where $\EX [X] \neq 0$, be such that
    \begin{equation}\label{eq:S_q-ass}
        \sup_{\eps>0}\EX [X_\eps^{-2}] < \infty.
    \end{equation}
    Then there exists $q$ that is independent of $\eps$ such that $X_{\eps}^{\diamond (-1)} \to X^{\diamond (-1)}$ in $(\mathcal{S})_{-1,-q}$.
\end{lemma}
\begin{proof}
    By~\cite[Proposition 3.1]{hu1994} we have $X_{\eps}^{\diamond (-1)} \to X^{\diamond (-1)}$ in $(\mathcal{S})_{-1,-q}$ for some $q$ if $(SX_\eps^{\diamond (-1)})(h) \allowbreak \to (SX^{\diamond (-1)})(h)$ and
    \begin{equation*}
        (SX_\eps^{\diamond (-1)})(h) \leq Ce^{K|h|_{2,p}^2}
    \end{equation*}
    for constants $C,K$ and $p>0$ that does not depend on $\eps$\footnote{For the precise definition and properties of the norm $|h|_{2,p}$, see e.g.~\cite{potthoff-streit1991}.}. The first assertion is clear from~\cref{eq:inverse-S} together with the fact $X_\eps \to X$ in $L^2_\mathcal{H}(\Omega)$. For the second assertion, note that by \textcolor{black}{Sedrakyan's inequality / Titu's lemma alongside the} Cauchy-Schwarz inequality we obtain\textcolor{black}{, as $h$ is a standard Gaussian and thus} $\EX [e^{\diamond h}] \textcolor{black}{= \EX[e^{h-\EX[h^2/2]}] = e^{-\frac{1}{2}} \EX [e^{1 \cdot h}] }  =1$, \textcolor{black}{that}
    \begin{equation*}
        \EX [X_\eps e^{\diamond h}] \geq \frac{\EX [e^{\diamond h}]}{\EX \left[\frac{e^{\diamond h}}{X_\eps}\right]} \geq \frac{\EX [e^{\diamond h}]}{\textcolor{black}{\sqrt{\EX [(e^{\diamond h})^2]}\sqrt{\EX [X_\eps^{-2}]}}} \geq Ce^{-\EX [h^2]},
    \end{equation*}
    where the last inequality follows from the assumption \textcolor{black}{in \eqref{eq:S_q-ass}} and the fact that $\EX[(e^{\diamond h})^2] = e^{\EX[h^2]}$, see~\cite[Corollary 3.37]{janson1997}.
    Hence, again by~\cref{eq:inverse-S}, we obtain
    \begin{equation*}
        (SX_\eps^{\diamond (-1)})(h) \leq Ce^{\EX [h^2]}
    \end{equation*}
    which completes the proof since $\EX [h^2] \leq |h|_{2,p}^2$ for any $p>0$, see \cite[pp. 214]{potthoff-streit1991}.
\end{proof}
\begin{proof}[Proof of~\cref{thm:wick-eqn-is-solvable}]
    We \textcolor{black}{prove this only in the case} of periodic boundary data, as again the other cases can be proved with similar arguments. Moreover, we set $U(T) = U(0)=0$ (recall that the solution is unique up to additive constant). In this case the approximated equation is given by
    \begin{align} \label{eq:wick-eqn-mollified}
        -\left(e^{\diamond (-X_{\beta,\eps}(t))}\diamond  U_\eps'(t)   \right)' = f(t), \qquad t \in (0,T)
    \end{align}
    with boundary conditions
    \begin{align*}
        U_\eps(0)                                             & =U_\eps(T) \nonumber \\
        e^{\diamond (-X_{\beta,\eps}(0))} \diamond U_\eps'(0) & =
        e^{\diamond (-X_{\beta,\eps}(T))} \diamond U_\eps'(T).
    \end{align*}
    Integrating once and Wick multiplying with the inverse $\left(e^{\diamond (-X_{\beta,\eps}(t))}\right)^{\diamond(-1)} = e^{\diamond X_{\beta,\eps}(t)}$ yields
    \begin{equation*}
        U_\eps'(t) =  \left(-\int_{0}^t f(s) \, ds + \kappa_\eps\right) \diamond e^{\diamond X_{\beta,\eps}(t)}
    \end{equation*}
    for some random variable $\kappa_\eps$.
    Integrating once more yields
    \begin{align}\label{eq:wick-periodic-soln}
        U_\eps (t) = \int_{0}^t (\kappa_\eps + F(s))\diamond e^{\diamond X_{\beta,\eps}(s)} \, ds,
    \end{align}
    where the \textcolor{black}{first} boundary condition implies that $\kappa_\eps$ is given by
    \begin{align*}
        \kappa_\eps  = -\left(\int_{0}^T F(s)e^{\diamond X_{\beta,\eps}(s)} \,ds \right)\diamond \left(\int_{0}^T e^{\diamond X_{\beta,\eps}(s)}\,ds\right)^{\diamond(-1)}.
    \end{align*}
    Taking the $S$-transform now gives
    \begin{equation*}
        (SU_{\eps}(t))(h) = (S\kappa_\eps)(h)\left(S \left [\int_0^t e^{\diamond X_{\beta,\eps}(s)}ds \right ]\right)(h) + \left(S \left [\int_0^t F(s)e^{\diamond X_{\beta,\eps}(s)}ds\right ] \right)(h).
    \end{equation*}
    Note that here we have, for all $t\in(0,T)$,
    \begin{equation*}
        \int_0^t e^{\diamond X_{\beta,\eps}(s)}ds \to \int_0^t e^{\diamond X_\beta(s)}ds
    \end{equation*}
    and
    \begin{equation*}
        \int_0^t F(s)e^{\diamond X_{\beta,\eps}(s)}ds \to \int_0^t F(s)e^{\diamond X_\beta(s)}ds
    \end{equation*}
    in $L^2_\mathcal{H}(\Omega)$. \textcolor{black}{We next recall the Kahane convexity / concavity inequality (see, for instance, ~\cite[Theorem 2.1]{rhodes2014}), which states that if $(A_i)_{1 \leq i \leq n}$ and $(B_i)_{1 \leq i \leq n}$ are two centered Gaussian vectors such that 
    \begin{align*}
        \forall i,j, \quad \EX [A_i A_j] \leq \EX [B_i B_j], 
    \end{align*}
    then for all combinations of nonnegative weights $(p_i)_{1 \leq i \leq n}$ and all convex (resp. concave) functions $F \colon \R_+ \to \R$ with at most polynomial growth at infinity, it holds that
    \begin{align}\label{eq:kahane-convex}
        \EX \left[ F \left(\sum_{i=1}^n p_i e^{A_i - \frac{1}{2} \EX [A_i^2]}    \right)   \right] \leq (\textnormal{resp.} \, \leq\, ) \EX \left[ F \left(\sum_{i=1}^n p_i e^{B_i - \frac{1}{2} \EX [B_i^2]}    \right)   \right].
    \end{align}
    Now, to apply the continuous version of \eqref{eq:kahane-convex}, we note that since $X_{\beta, \eps}$ is a regularization of $X_\beta$, we have $\EX[X_{\beta,\eps}(t)X_{\beta,\eps}(s)] \leq \EX[X_{\beta}(t)X_{\beta}(s)]$, and by choosing $F(x)=x^{-p}, p < 0$, we obtain for all $\eps > 0:$
    \begin{align*}
      \EX \left[ \left(\int_{0}^T e^{X_{\beta, \eps}(t) - \frac{1}{2} \EX [X_{\beta, \eps}^2]} \, dt   \right)^{-p}   \right]  \leq \EX \left[ \left(\int_{0}^T e^{X_{\beta}(t) - \frac{1}{2} \EX [X_{\beta}(t)^2]}   \, dt  \right)^{-p}    \right].
    \end{align*}
 We know that above the quantity on the right-hand side is finite by Theorem \ref{thm:GMC-moments}. Thus t}he random variable $\int_0^T e^{\diamond X_{\beta,\eps}(s)}ds$ has all negative moments finite, uniformly in $\eps$ (\textcolor{black}{one also obtains this} by adding smoothing to the proof in~\cite[Appendix B]{astala-et-al}). By~\cref{lemma:S_q-convergence}, it follows that
    there exists $q$ large enough, independent of $\eps$, such that
    \begin{equation*}
        \kappa_\eps  = -\left(\int_{0}^T F(s)e^{\diamond X_{\beta,\eps}(s) \, ds}\right)\diamond \left(\int_{0}^T e^{\diamond X_{\beta,\eps}(s)\,ds}\right)^{\diamond(-1)} \in (\mathcal{S})_{-1,-q}.
    \end{equation*}
    This further implies that the solution $U_\eps(t) \in (\mathcal{S})_{-1,-q}$ for large enough $q$ which does not depend on $\eps$, and in particular this determines the class of test points $h$ given by the condition $e^{\diamond h} \in (\mathcal{S})_{1,q}$. Now passing to the limit and using the above convergences in $L^2_\mathcal{H}(\Omega)$ (that implies convergence of the $S$-transforms on any test point $h$) yields, for any $h$ such that $e^{\diamond h} \in (\mathcal{S})_{1,q}$ where now $q$ is independent of $\eps$,
    \begin{equation*}
        \lim_{\eps\to 0} (SU_{\eps}(t))(h) = (S\kappa)(h)\left(S\int_0^t e^{\diamond X_{\beta}(s)}ds\right)(h) + \left(S\int_0^t F(s)e^{\diamond X_{\beta}(s)}ds\right)(h),
    \end{equation*}
    where
    \begin{equation*}
        \kappa =  -\left(\int_{0}^T F(s)e^{\diamond X_{\beta}(s) \, ds}\right)\diamond \left(\int_{0}^T e^{\diamond X_{\beta}(s)\,ds}\right)^{\diamond(-1)} \in (\mathcal{S})_{-1,-q}.
    \end{equation*}
    This provides us the claimed $U(t)$ as the solution. Finally, the uniqueness follows from the invertibility of the $S$-transform. This completes the proof.
\end{proof}

\begin{supplementary}
    Computing the approximate covariance on the limit:

    \begin{align*}
        \EX [X_\beta^\eps(t) X_\beta(e)] & =\EX \left[ X_\beta^\eps (t) \int_{0}^T e(a) \, dX_\beta(a)\right]   \\
                                        & =
        \EX \left[ \int_{0}^T \varphi_\eps(t-y) X_\beta(y) \, dy   \int_{0}^T e(a) \, X_\beta(a) \,da  \right] \\
                                        & =
        \int_{0}^T \int_{0}^T \varphi_\eps(t-y) e(a) \EX [X_\beta(y) X_\beta(a)] \, dy \, da                   \\
                                        & =
        -\beta^2 \int_{0}^T \int_{0}^T \varphi_\eps(t-y) e(a) \log\abs{y-a} \, dy \, da =: R_\eps(t)
    \end{align*}
    and we have
    \begin{align*}
        \lim_{\eps \to 0} R_\eps(t) =-\beta^2 \int_{0}^T e(a) \log\abs{t-a} \, da.
    \end{align*}

\end{supplementary}

\section{\texorpdfstring{$S$}{S}-transform at \texorpdfstring{$X_{\beta}$}{X} and projections into a subspace}\label{sec:projection}
By~\cref{thm:non-wick-eqn-is-solvable} and~\cref{thm:wick-eqn-is-solvable}, the solutions to problems~\cref{eq:non-wick-equation-general-problem} and~\cref{eq:wick-equation-general-problem} respectively are very similar, and essentially generated by random variables of the form
\begin{equation*}
    \int_0^T \varphi(y,s)e^{\diamond X_{\beta}(s)}ds
\end{equation*}
for some suitable $\varphi(y,s)$. Motivated by this, we consider projections into a subspace $L_{\beta'}^2(\Omega) \subset L^2(\Omega)$ generated by random variables of form
\begin{equation*}
    \int_0^T \varphi(y,s)e^{\diamond X_{\beta'}(s)}ds
\end{equation*}
\textcolor{black}{for fixed $y$}, with appropriately chosen $\beta' < 1$. Choosing $\beta'\neq \beta$, corresponds to projecting the solution into a subspace with less / higher order moments, cf.~\cref{thm:GMC-moments}. To this end, we introduce the $S$-transform at ``a point'' $X_{\beta'}(s)$ directly without mollification, and show that these $S$-transforms characterize our subspace $L^2_{\beta'}(\Omega)$, defined as the subspace generated by the GMC measure $d\mu_{\beta'}$.

For the sake of generality and the purposes of \textcolor{black}{comparison with an earlier paper (see \cite{renormalized})} studying the $d$-dimensional case, we introduce our definitions concerning log-correlated fields and GMC measures on a smooth bounded domain $D \subset \mathbb{R}^d$ from which we naturally recover our one-dimensional case by choosing $D = (0,T)$. In the general $d$-dimensional case, we consider values $\beta < \sqrt{d}$ corresponding to the $L^2$-range of the GMC measure. In the sequel, we always assume $D$ to be a \textcolor{black}{connected domain with a smooth boundary}. We also need to define spaces $W^{s,2}(D)$ and its topological dual $W_0^{-s,2}(D)$:
\begin{definition}
\label{def:frac-sobolev-general}
    Let $s \in \R$, then $W^{s,2}(\R^d)$ is defined as
    \begin{equation*}
        W^{s,2}(\R^d) = \{f \in S'(\R^d) : \mathcal{F}{f}\in L^1_{loc}(\R^d) \textnormal{ s.t. } \Vert f\Vert_{W^{s,2}(\R^d)}^2 =  \int_{\R^d} (1+|\xi|^2)^s |\mathcal{F}{f}(\xi)|^2d\xi < \infty\},
    \end{equation*}
    where $\mathcal{F}{f}$ denotes the Fourier transform and $S'$ denotes the space of Schwartz distributions. 
    
\textcolor{black}{For $s\geq 0$, the } space $W^{s,2}(D)$ is defined as restrictions $f_{|D}$ of functions $f\in W^{s,2}(\R^d)$ into $D$, equipped with the norm
    \begin{equation}
    \label{eq:frac-sobolev-norm-restricted}
        \Vert f\Vert_{W^{s,2}(D)} = \inf\{\Vert g\Vert_{W^{s,2}(\R^d)}, g_{|D} = f\}\textcolor{black}{.}
    \end{equation}
\end{definition}
\textcolor{black}{We denote by $W_0^{-s,2}(D)$ the} topological dual of $W^{s,2}(D)$ (in the standard distributional duality sense). 
We \textcolor{black}{will} consider the following subspace of $L^2(\Omega)$.
\begin{definition}[$L_\beta^2(\Omega)$ space generated by GMC $d\mu_\beta$]\label{def:L2-GMC}
    Let $\beta\in(0,\sqrt{d})$. The space $L_{\beta}^2(\Omega) := L^2(\Omega, \mathbb{P},\mathcal{F}_{d\mu_{\beta}})$ generated by the GMC measure $d\mu_\beta$ is the $L^2(\Omega)$-closure of the linear space spanned by random variables of the form
    \begin{equation}
        \label{eq:basic-RV}
        \int_{D} \varphi(z) d\mu_\beta(z), \quad \varphi \in C_{\textcolor{black}{c}}^\infty(D).
    \end{equation}
\end{definition}
In order to characterize the random variables \textcolor{black}{belonging to} $L_\beta^2(\Omega)$ and define the $S$-transform at $X_\beta(s)$, we introduce the operator acting on $\phi:D \to \R$
\begin{align}\label{eq:operator-G}
    G\phi(z) = G_{\alpha}\phi(z): & =\int_D\exp (\alpha\EX [X(z)X(y)])\phi(y)\, dy \; \nonumber \\
                                  & =
    \int_D|z-y|^{-\alpha}e^{\alpha g(z,y)}\phi(y) dy,
\end{align}
where $\alpha\in (0,d)$ and $X$ is a log-correlated Gaussian field as in~\cref{def:log-correlated-field}. Note that if $g\equiv 0$ in~\cref{eq:general-cov}, then the operator $G$ reduces to the Riesz potential on $D$. We will later on apply the operator $G_\alpha$ with the choices $\alpha \in \{\beta^2,\beta'^2,\beta\beta'\}$.

For the invertibility of the $S$-transform, we need to assume that the log-correlated field has the following additional property.
\begin{definition}\label{def:X} We say that our log-correlated field $X_\beta := \beta X$  is \emph{$\beta$-non-degenerate} on a bounded domain $D \subset \mathbb{R}^d$ if it satisfies the following property: there is a constant $c_0>0$ such that for each   $\psi\in C^\infty_{\textcolor{black}{c}}(D)$ it holds that
    \begin{equation}\label{eq:full}
        \int_{D} \psi(z)G_{\beta^2}\psi(z)\,dz\geq c_0\|\psi\|^2_{W_0^{-{s_\beta},2}(D)},
    \end{equation}
    where $s_\beta = \frac{d-\beta^2}{2}$. 
\end{definition}
Our next result verifies that the class of such fields is not empty.
\begin{lemma}\label{le:non-degenerate}
    There are log-correlated fields $X$ on $D$  such that the fields $X_\beta:=\beta X$  are $\beta$-non-degenerate on $D$ for all $\beta\in (0,\sqrt{d})$.
\end{lemma}

\begin{proof}
    Let $\psi_0\in C^\infty_{\textcolor{black}{c}}(\R^d)$ be a positive cut-off function equal to $1$ on $B(0,4R)$, where $R:={\rm diam \,}(D)$. Since the Fourier-transform of $\log (1/|\cdot|)$ is equal (up to a multiplicative constant) to $|\xi|^{-d}$ for large $\xi$, we infer that the Fourier-transform of $\psi_0\log (1/|\cdot|)$ is a smooth function that is positive outside a suitable neighborhood of the origin. We then choose a radially symmetric and positive $r\in C_{\textcolor{black}{c}}^\infty (\R^d)$ such that $\mathcal{F}{\left(\psi_0\log (1/|\cdot|)\right)}+r\geq 0$ everywhere. Then
    \begin{align*}
        \psi_0(z-y)\log (1/|z-y|)+h(z-y),\qquad \textnormal{where}\quad  h:=\mathcal{F}^{-1}{r},
    \end{align*}
    is a covariance  on $\R^d$, and we denote  by $Y$ the Gaussian distribution valued field with this covariance.  We claim that the field $X:=Y_{|D}$ is the field we are looking for.

    For that end, by using the assumption on the support of  $\psi_0$  we note that the operator $G := G_{\beta^2}$ for the field $X$ has the kernel  $H(z-y)$ ($z,y\in D$), where
    \begin{align*}
        H(z)\; :\; =e^{\beta^2r(z)}|z|^{-\beta^2}.
    \end{align*}
    We extend $H$ to all of $\R^d$ by this formula and note that
    \begin{align*}
        a:= \mathcal{F}{\left(e^{\beta^2r(z)}-1\right)}=\sum_{n=1}^\infty\frac{\beta^{2n} r^{\ast n}}{n!}
    \end{align*}
    where $r^{\ast n} = r \ast r \ast \cdots \ast r$ is the $n$-fold convolution. Thus, $a\geq 0$ and $a$ is a Schwarz function as it is the Fourier transform of a test function. Thus,
    \begin{align*}
        \mathcal{F} H=|\cdot|^{-2s_\beta} + a \ast |\cdot|^{-2s_\beta}
    \end{align*}
    that decays as $|\cdot|^{-2s_\beta}$. The claim follows, since if $\psi\in C^\infty_{\textcolor{black}{c}}(D)$, we obtain by the previous estimate and Plancherel identity that for some constant $c$ independent of $\psi$ that
    \begin{align*}
        \int_{D} \psi(z)(G\psi)(z)\,dz \; =\; \int_{\R^d}\mathcal{F} H(\xi)|\mathcal{F}\psi(\xi)|^2\, d\xi\;  \geq c \; \|\psi\|^2_{W_0^{-s_\beta,2}(\R^d)},
    \end{align*}
    where the last inequality follows as $\mathcal{F}{H}(\xi)\geq c |\xi|^{-2s_\beta}$.

\end{proof}
\begin{lemma}\label{le:inversion}
    Let $\alpha\in(0,d)$. Assume that the domain $D$ is smooth and bounded and the field $X$ is $\beta$-non-degenerate. Then the  mapping $G_{\alpha}$ given in~\cref{eq:operator-G} extends to a bijective map
    \begin{align*}
        G: W_0^{-\frac{d-\alpha}{2},2}(D)\to  W^{\frac{d-\alpha}{2},2}(D).
    \end{align*}
\end{lemma}

\begin{proof}
    Let $s_\alpha = (d-\alpha)/2$. We first verify that $G$ is bounded between the stated spaces. For that end, let $G_0$ denote the operator with the kernel $H_0(z-y)$, where
    \begin{align*}
        H_0(z):= |z|^{2s_\alpha -d}\psi_0(z),
    \end{align*}
    with $\psi_0$ as in the previous lemma. Now
    \begin{align*}
        \mathcal{F} H_0=\mathcal{F} \psi_0 * |\cdot|^{-2s_\alpha} \lesssim (1+|\cdot|^2)^{-s_\alpha},
    \end{align*}
    \textcolor{black}{where $\lesssim$ denotes an inequality that holds up to an unimportant multiplicative constant.} Since $G_0$ is a convolution operator, we immediately obtain that $G_0: W_0^{-s_\alpha,2}(D)\to  W^{s_\alpha,2}(D)$. In order to consider the operator $G$, denote for $\lambda\in\R^d$  by $M_\lambda$ the multiplication operator $f\mapsto e^{i\lambda\cdot  x}f$. The norm of $M_\lambda$ on both of the spaces  $W_0^{-s_\alpha,2}(D)$
    and $W^{s_\alpha,2}(D)$ is bounded by $c(1+|\lambda|)^a$ with some fixed $a$, since on the Fourier side $M_\lambda$ corresponds to translation by $\lambda$. The action of $G$ in our situation is given by the kernel
    $k(z,y)H_0(z-y)$, where $k(z,y):= \psi_0(z)\psi_0(y)e^{\alpha g(z,y)}$. The boundedness of $G$ is then seen by writing
    \begin{align*}
        G\;=\; G_0+ \int_{\R^d\times\R^d}
        \mathcal{F} k
        (\lambda_1,\lambda_2)\; M_{\lambda_1}G_0M_{\lambda_2}\; d\lambda_1d\lambda_2
    \end{align*}
    and use the fact  that $\mathcal{F} k$ is a Schwarz test function.

    Now the rest of the statement follows easily just by the non-degeneracy of $X_{\sqrt{\alpha}}$ and the fact that because $D$ is smooth, the spaces $W_0^{-s_\alpha,2}(D)$ and  $W^{s_\alpha,2}(D)$ are dual spaces of each others with respect to the $L^2$-pairing. Namely, the non-degeneracy condition implies that $G: W_0^{-s_\alpha,2}(D)\to  W^{s_\alpha,2}(D)$ is lower bounded. Hence, it is injective and its image, call it $M$, is a closed subspace of $W^{s_\alpha,2}(D)$. If $M$ would not be all of  $W^{s_\alpha,2}(D)$, we could find a  non-zero dual element $\psi\in W_0^{-s_\alpha,2}(D)$ which \textcolor{black}{annihilates} all elements in $M$. Especially $\psi$ should annihilate $G\psi$, but this contradicts the non-degeneracy assumption.
\end{proof}
\begin{remark}
\label{remark:general-mapping}
Note that by the above arguments we actually see that $G_\delta: W_0^{\alpha + \delta - d,2}(D) \to W^{\alpha,2}(D)$ for any $\delta\in (0,d)$ and $\alpha>0$ such that $\alpha + \delta<d$. Indeed, when $G_\delta$ is the convolution operator with kernel $H_\delta$, we have 
$\mathcal{F}(G_\delta u) = \mathcal{F}H_\delta \mathcal{F}u$, where 
$$
\mathcal{F}H_\delta \lesssim (1+|\cdot|^2)^{-\frac{d-\delta}{2}}.
$$
Thus we have 
\begin{align*}
    \int_{\mathbb{R}^d} (1+|\xi|^2)^\alpha |\mathcal{F}(G_\delta u)(\xi)|^2d\xi
    =&   
    \int_{\mathbb{R}^d} (1+|\xi|^2)^\alpha |\mathcal{F} H_\delta (\xi)\mathcal{F} u(\xi)|^2d\xi\\
    \\
    \lesssim &
    \int_{\mathbb{R}^d} (1+|\xi|^2)^{\alpha+\delta-d} | \mathcal{F} u(\xi)|^2d\xi
\end{align*}
which is finite for any $u \in W_0^{\alpha+\delta-d,2}(\mathbb{R}^d)$ by \cref{def:frac-sobolev-general}. Taking into account \cref{eq:frac-sobolev-norm-restricted}, this implies that  $G_\delta$ is a bounded operator from $W_0^{\alpha + \delta - 1,2}(D)$ into $W^{\alpha,2}(D)$. Finally, the case of general operator $G$ follows by the arguments of the proof of \cref{le:inversion}. 
\end{remark}
We obtain immediately the following result characterizing random variables in $L^2_\beta(\Omega)$.
\begin{cor}\label{cor:L2beta-characterisation}
    Let $D \subset \mathbb{R}^d$ be a smooth and bounded domain and let $\beta\in(0,\sqrt{d})$. Suppose $X$ is a $\beta$-non-degenerate and let $Z \in L^2_\beta(\Omega)$. Then $Z$ has the form
    \begin{equation*}
        Z = \int_D\varphi(z)d\mu_\beta(z),
    \end{equation*}
    for some $\varphi \in W_0^{-s_\beta,2}(D)$, where $s_\beta= \frac{d-\beta^2}{2}$.
\end{cor}
\begin{proof}
    \textcolor{black}{First,} for $\textcolor{black}{Y} = \int_{D} \textcolor{black}{\psi(z)} d\mu_\beta(z)$ with $\textcolor{black}{\psi} \in C_{\textcolor{black}{c}}^\infty(D)$ we have, using \eqref{eq:wick-exponential-sum-identity}, that
    \begin{align*}
        \EX [\textcolor{black}{Y}^2] &=  \textcolor{black}{\int_D \int_D \psi(z)\psi(y) \EX\left[e^{\diamond X_\beta(z)} e^{\diamond X_\beta(y)}  \right] dy \, dz} \\
        & = 
        \int_D \int_D \textcolor{black}{\psi}(z)\textcolor{black}{\psi}(y)\exp(\beta^2 \EX [X(\textcolor{black}{z})X(\textcolor{black}{y})])  dy \, dz \\
        & = 
        \int_D \textcolor{black}{\psi}(z)G_{\beta^2}\textcolor{black}{\psi}(z) dz \\
        & = 
        \textcolor{black}{ \langle\textcolor{black}{\psi}, G_{\beta^2} \textcolor{black}{\psi} \rangle_{W_0^{-s_{\beta},2}(D), W^{s_{\beta},2}(D)}, }
    \end{align*}
    \textcolor{black}{where the last notation refers to the usual dual pairing. By the $\beta$-non-degeneracy assumption in Definition \ref{def:X} and the continuity of $G_{\beta^2}$, this pairing defines a norm that is equivalent to the usual $W_0^{-s_\beta,2}(D)$ norm.} The claim now follows by taking the closure in $L^2(\Omega)$.
\end{proof}
We are now ready to define the $S$-transform of square integrable random variables at a point $X_\beta(z)$. For this, observe that for $V \in L^2_{\beta}(\Omega)$ of form $V = \int_{D} \psi(z) d\mu_\beta(z)$, where $\psi \in W_0^{-s_\beta,2}(D)$ with $s_\beta= \frac{d-\beta^2}{2}$, formal computations yield
\begin{equation*}
    ZV = \int_D \psi(z)Ze^{\diamond X_\beta(z)}dz
\end{equation*}
for any $Z \in L^2(\Omega)$ from which taking expectation gives, formally,
\begin{equation*}
    \EX [ZV] = \int_D \psi(z)(SZ)(X_\beta(z))dz.
\end{equation*}
More precisely, the mapping $\psi \mapsto \EX [ZV]$ is a continuous functional on $W^{-s_\beta,2}_0(D)$. By duality, there exists $\iota \in W^{s_\beta,2}(D)$ such that
\begin{equation*}
    \EX [ZV] = \int_D \psi(z) \iota(z)dz.
\end{equation*}
This leads to the following definition.
\begin{definition}[$S$-transform in $L^2_\beta(\Omega)$]\label{def:S-transform-log}
    Let $Z \in L^2(\Omega)$ and set $s_\beta= \frac{d-\beta^2}{2}$. For $z\in D$ we identify $(SZ)(X_\beta(z))$ as the (unique) element of $W^{s_\beta,2}(D)$ defined through duality
    \begin{equation*}
        \EX [ZV] = \int_D \psi(z)(SZ)(X_{\beta}(z))dz
    \end{equation*}
    for all $V = \int_{D} \psi(z) d\mu_\beta(z) \in L^2_{\beta}(\Omega)$.
\end{definition}
Note that the definition makes sense for almost every $z$ which is sufficient for our purposes. The following result allows us to characterize the subspace $L^2_\beta(\Omega)$ (or projections into it) via the $S$-transform at points $X_\beta(z)$.
\begin{lemma}
    \textcolor{black}{Assume that the domain $D$ is smooth and bounded and the field $X$ is $\beta$-non-degenerate.}  Let $Z \in L^2(\Omega)$ with a decomposition $Z = Z_\beta + Z'_\beta$, where $Z_\beta = \int_D \varphi(s)d\mu_\beta(s) \in L^2_\beta(\Omega)$ is the orthogonal projection into the subspace $L^2_\beta(\Omega) \subset L^2(\Omega)$. Then we have
    \begin{equation*}
        (SZ_\beta)(X_\beta(z)) = G_{\beta^2}\varphi(z)
    \end{equation*}
    and
    \begin{equation*}
        (SZ'_\beta)(X_\beta(z)) = 0.
    \end{equation*}
\end{lemma}
\begin{proof}
    The claim $(SZ'_\beta)(X_\beta(z)) = 0$ follows immediately from~\cref{def:S-transform-log} by observing that then
    \begin{equation*}
        \int_D \psi(z)(SZ'_\beta)(X_\beta(z)) dz = 0 \quad \forall \psi \in W_0^{-s_\beta,2}(D).
    \end{equation*}
    For $Z_\beta$, obtain that
    \begin{equation*}
        \EX [Z_\beta V] = \EX \left[\int_D \int_D \psi(z)\varphi(y)d \mu_\beta(z)d\mu_\beta(y)\right] = \int_D \psi(z) G_{\beta^2}\varphi(z) dz,
    \end{equation*}
    where now, by~\cref{le:inversion}, $G_{\beta^2}\varphi(z) \in W^{s_\beta,2}(D)$. This completes the proof.
\end{proof}
This leads to the following characterization result of random variables in $L^2_\beta(\Omega)$ via the $S$-transform:
\begin{lemma}\label{lemma:S-uniqueness}
    \textcolor{black}{Assume that the domain $D$ is smooth and bounded and the field $X$ is $\beta$-non-degenerate.} The $S$-transform at points $X_\beta(z)$ determine random variables $Z \in L^2_\beta(\Omega)$ uniquely.
\end{lemma}
\begin{proof}
    Suppose that for random variables $Z_1 = \int_D \varphi_1(y)d\mu_\beta(y)$ and $Z_2 = \int_D \varphi_2(y)d\mu_\beta(y)$ we have $(SZ_1)(X_\beta(z)) = (SZ_2)(X_\beta(z))$ for almost every $z\in D$. Then $G_{\beta^2}\varphi_1(z) = G_{\beta^2}\varphi_2(z)$ for almost every $z\in D$ from which $\varphi_1=\varphi_2$ follows by~\cref{le:inversion}.
\end{proof}

We now leave the general $d$-dimensional setting and focus on the case $d=1$ with $D=(0,T)$. Hence, in the sequel, we have $s_{\beta}=\frac{1-\beta^2}{2}$. By using the $S$-transform, we are able to give explicit expressions for the projections into $L^2_{\beta'}(\Omega)$ of random variables in $L^2_\beta(\Omega)$, for $0< \beta, \beta' < 1$.
\begin{prop}\label{prop:projection}
    Let $\beta,\beta' \in (0,1)$ and let
    \begin{equation*}
        Z = \int_D \varphi(z)d\mu_\beta(z) \in L^2_\beta(\Omega)
    \end{equation*}
    for some $\varphi \in W_0^{-\frac{1-\beta^2}{2},2}(0,T)$. Then the projection of $Z$ into the subspace $L_{\beta'}^2(\Omega)$ is given by
    \begin{equation*}
        Z_{\beta'} = \int_0^T \varphi_{\beta'}(z)d\mu_{\beta'}(z),
    \end{equation*}
    where $\varphi_{\beta'} \in W_0^{-\frac{1-(\beta')^2}{2},2}(0,T)$ solves
    $G_{\beta\beta'}\varphi = G_{(\beta')^2}\varphi_{\beta'}$.
\end{prop}
\begin{proof}
    Let $Z = \int_0^T \varphi(z)d\mu_\beta(z) \in L^2_\beta(\Omega)$ and $V = \int_0^T \psi(z)d\mu_{\beta'}(z) \in L^2_{\beta'}(\Omega)$ be a test random variable. It is enough to check that $\EX [Z_{\beta'}V] = \EX [ZV]$ for each $\psi \in W_0^{-\frac{1-(\beta')^2}{2},2}(0,T)$. Here we have
    \begin{equation*}
        \EX [Z_{\beta'}V] = \int_0^T \psi(z)G_{(\beta')^2}\varphi_{\beta'}(z)dz.
    \end{equation*}
    Similarly,~\cref{def:S-transform-log} yields
    \begin{equation*}
        \EX [ZV] = \int_0^T \psi(z)(SZ)(X_{\beta'}(z))dz 
    \end{equation*}
    for $(SZ)(X_{\beta'}(z)) \in  W^{\frac{1-(\beta')^2}{2},2}(0,T)$ while formal computations as above gives
    \begin{equation}
        \label{eq:beta'-projection}
        \EX [ZV] = \int_0^T \psi(z)G_{\beta\beta'}\varphi(z)dz,
    \end{equation}
    where now $\psi \in W_0^{-\frac{1-(\beta')^2}{2},2}(0,T)$ and $\varphi \in  W_0^{-\frac{1-\beta^2}{2},2}(0,T)$. 
    We now claim that
    \begin{align}\label{eq:claim_projection}
        G_{\beta\beta'}\varphi(z) \in W^{\frac{1-(\beta')^2}{2},2}(0,T).      
    \end{align}
    Using claim~\cref{eq:claim_projection}, we see that~\cref{eq:beta'-projection} follows, since from~\cref{def:S-transform-log}, the uniqueness of the $S$-transform gives $(SZ)(X_{\beta'}(z)) = G_{\beta\beta'}\varphi(z)$. Now~\cref{le:inversion} gives that $G_{(\beta')^2}$ is invertible and hence $\EX [Z_{\beta'} V] = \EX [Z V]$ holds if and only if $G_{\beta \beta'} \varphi = G_{(\beta')^2} \varphi_{\beta'}$. For the claim~\cref{eq:claim_projection} note that, by Remark \ref{remark:general-mapping}, we have $G_{\beta\beta'}:W_0^{\alpha + \beta\beta' - 1,2}(0,T) \to W^{\alpha,2}(0,T)$ for any $\alpha\in (0,1-\beta\beta')$. In particular, this holds for $\alpha = \frac{1}{2} - \frac{(\beta')^2}{2}$ yielding 
    $G_{\beta\beta'}:W_0^{-\frac12 - \frac{(\beta')^2}{2}+ \beta\beta',2}(0,T) \to W^{\frac{1}{2} - \frac{(\beta')^2}{2},2}(0,T)$. Now \cref{eq:claim_projection} follows by noting that $W_0^{-\frac{1-\beta^2}{2}}(0,T) \subset W_0^{-\frac12 - \frac{(\beta')^2}{2}+ \beta\beta',2}(0,T)$ since
    $$
    -\frac12 - \frac{(\beta')^2}{2}+ \beta\beta' \leq -\frac{1-\beta^2}{2}
    $$
    is trivially true.  Hence,~\cref{eq:claim_projection} is proved.

\end{proof}
We now address projections of solutions to~\cref{eq:wick-equation-general-problem} into $L^2_{\beta'}(\Omega)$, for some given $\beta'\geq \beta$. As the solution given by~\cref{thm:wick-eqn-is-solvable} is a proper element of $L^2(\Omega)$ only in the case of initial value problem or with Neumann boundary condition, we define the (generalized) projection via the corresponding $S$-transformed deterministic equation, in the spirit of~\cref{def:wick-solution}. To this end, recall that by taking the $S$-transform in~\cref{eq:wick-equation-general-problem} on a test point $h$ gives us, after interchanging the order of expectation and differentiation, a deterministic equation
\begin{equation*}
    \left(e^{-\EX [hX_\beta(\cdot)]}u_h'(\cdot)\right)' = f(\cdot),
\end{equation*}
where $u_h(\cdot) = SU(\cdot)(h)$ is the $S$-transform of the stochastic solution $U$ tested at $h$.
Similarly, the terms $e^{\diamond (-X_\beta(0))} \diamond U'(0)$ and $e^{\diamond (-X_\beta(T))} \diamond U'(T)$ in the boundary conditions are translated into
$e^{-\EX [hX_\beta(0)]} u_h'(0)$ and $e^{-\EX [hX_\beta(T)]} u_h'(T)$. 
Hence, assuming that $U(y) \in L^2(\Omega)$ and if we are interested in the projection into $L^2_{\beta'}(\Omega)$, we would formally choose $h = X_{\beta'}(z)$. This motivates the following definition.
\begin{definition}\label{def:projection}
    For any solution to~\cref{eq:wick-equation-general-problem}, its projection into $L^2_{\beta'}(\Omega)$ for $\beta'\in(0,1)$ is a random function $U(y)$ such that $U(y) \in L^2_{\beta'}(\Omega)$ for all $y\in(0,T)$
    and $\tilde{u}_z(y) = (SU(y))(X_{\beta'}(z))\in W^{s_{\beta'},2}(0,T)$ solves the deterministic equation
    \begin{equation}
        \label{eq:projection-det-equation}
        \left(e^{-\beta\beta'\EX [X(z)X(\cdot)]}\tilde{u}'_z(\cdot)\right)' = f(\cdot)
    \end{equation}
    with the corresponding boundary terms given by
    \begin{enumerate}
        \item \emph{Initial data:} 
        \begin{equation*}
            \tilde{u}_z(0) = U_1 \quad \textnormal{and} \quad e^{-\beta\beta'\EX [X(z)X(0)]}\tilde{u}'_z(0) = U_2,
        \end{equation*}
        \item \emph{Dirichlet data:} 
        \begin{equation*}
            \tilde{u}_z(0) = U_1 \quad \textnormal{and} \quad \tilde{u}_z(T) = U_2,
        \end{equation*}
        \item \emph{Neumann boundary conditions:} 
        \begin{equation*}
            e^{-\beta\beta'\EX [X(z)X(0)]}\tilde{u}'_z(0) = U_1 \quad \textnormal{and} \quad    e^{-\beta\beta'\EX [X(z)X(T)]}\tilde{u}'_z(T) = U_2,
        \end{equation*}
        \item \emph{Periodic boundary conditions:} 
        \begin{equation*}
            \tilde{u}_z(0) = \tilde{u}_z(T) \quad \textnormal{and} \quad e^{-\beta\beta'\EX [X(z)X(0)]}\tilde{u}'_z(0) = e^{-\beta\beta'\EX [X(z)X(T)]}\tilde{u}'_z(T).
        \end{equation*}
    \end{enumerate}
\end{definition}
\begin{remark}
    \label{remark:comparison}
    It turns out that this definition makes sense. Indeed, in the case of the initial value problem or with Neumann boundary conditions, the solution to~\cref{eq:wick-equation-general-problem} is actually an element of $L^2_\beta(\Omega)$ with suitable kernel $\varphi$, and hence we can apply~\cref{prop:projection} to compute the projection into $L^2_{\beta'}(\Omega)$ that corresponds to the one obtained through~\cref{def:projection}, cf.~theorem below. In particular for $\beta'=\beta$, the $S$-transformed equation~\cref{eq:projection-det-equation} evaluated at $X_\beta(z)$ yields precisely the solutions given in~\cref{thm:wick-eqn-is-solvable}.
\end{remark}
\begin{thm}\label{thm:projection}
    \textcolor{black}{Assume that the field $X$ is $\beta$-non-degenerate, and } let $\beta<1$ and $\beta'\in [\beta,1)$. Then for any solution $U$ to~\cref{eq:wick-equation-general-problem} and any $y\in(0,T)$, the projection of $U$ into $L^2_{\beta'}(\Omega)$ (in the sense of~\cref{def:projection}) is given by
    \begin{equation*}
        U(y) = \int_0^T \varphi(y,a)d\mu_{\beta'}(a)
    \end{equation*}
    with
    \begin{equation*}
        \varphi(y,a) = G_{\beta'^2}^{-1}[\tilde{u}_\cdot(y)](a),
    \end{equation*}
    and $\tilde{u}_z(y)$ is the solution to the deterministic problem~\cref{eq:projection-det-equation}.
\end{thm}
The proof makes use of the following elementary lemma. For the reader's convenience, we present a short proof.
\begin{lemma}
    \label{lemma:simple}
    Let $s\in (0,1)$ and let $T<\infty$ be fixed. Let $f,g\in W^{s,2}(0,T)$.
    \begin{itemize}
        \item If $f,g\in L^\infty(0,T)$, then $fg \in W^{s,2}(0,T)$.
        \item If $f,1/g \in L^\infty(0,T)$, then $f/g \in L^\infty(0,T) \cap W^{s,2}(0,T)$.
    \end{itemize}
\end{lemma}
\begin{proof}
    Recall that the norm of $W^{s,2}(0,T)$ can be given as
    \begin{equation*}
        \Vert f\Vert_{W^{s,2}(0,T)} = \Vert f\Vert_{L^2(0,T)} + \left(\int_0^T\int_0^T \frac{|f(x)-f(y)|^2}{|x-y|^{1+2s}}dxdy\right)^{1/2}.
    \end{equation*}
    Since $[0,T]$ is compact, we have $L^\infty(0,T) \subset L^2(0,T)$ and hence it suffices to consider the seminorm part. By using triangle and Jensen's inequality
    \begin{equation*}
        |f(x)g(x)-f(y)g(y)|^2 \leq 2\left[|f(x)|^2|g(x)-g(y)|^2 + |g(y)|^2|f(x)-f(y)|^2\right]
    \end{equation*}
    gives us the first claim, since $f,g\in L^\infty(0,T) \cap W^{s,2}(0,T)$. Similarly, for the second claim, we write
    \begin{equation*}
        \left\vert\frac{f(x)}{g(x)} - \frac{f(y)}{g(y)}\right\vert \leq \frac{|f(x)-f(y)|}{|g(x)|} + \frac{|f(y)|}{|g(x)g(y)|}|g(x)-g(y)|
    \end{equation*}
    from which the claim follows since $f,1/g \in L^\infty(0,T)$ and $f,g\in W^{s,2}(0,T)$.
\end{proof}
\begin{proof}[Proof of~\cref{thm:projection}]
    It is a routine exercise to check that the unique solution to equation~\cref{eq:projection-det-equation} is given by
    \begin{equation*}
        \tilde{u}_z(t) = \tilde{u}_z(0) + \int_0^t \left[e^{\beta\beta'\EX [X(z)X(y)]}e^{-\beta\beta'\EX [X(z)X(0)]}\tilde{u}_z'(0) + e^{\beta\beta'\EX [X(z)X(y)]}F(y)\right] dy,
    \end{equation*}
    where $u(0)$ and $\tilde{u}_z'(0)$ are determined by the boundary/initial conditions. Hence, in order to prove the result it suffices to show that this solution is an element of $W^{s_{\beta'},2}(0,T)$ w.r.t.~the test point $z$. Indeed, then it follows that the projection of $U(t)$ into $L^2_{\beta'}(\Omega)$\textcolor{black}{, in the sense of \cref{{def:projection}}}, has a representation \textcolor{black}{(by Definition \ref{def:S-transform-log} and applying Corollary \ref{cor:L2beta-characterisation} alongside Lemma \ref{le:inversion})}
    \begin{equation*}
        U_{\beta'}(t) = \int_0^T \left[G_{\beta'^2}^{-1}\tilde{u}_\cdot(t)\right](z)d\mu_{\beta'}(z).
    \end{equation*}
       Observe now that $\tilde{u}_z(t)$ can be written as
    \begin{equation}\label{eq:tilde_uz}
        \tilde{u}_z(t) = \tilde{u}_z(0) + e^{-\beta\beta'\EX [X(z)X(0)]}\tilde{u}_z'(0)(G_{\beta\beta'}I_{[0,t]})(z) + (G_{\beta\beta'}[FI_{[0,t]}])(z).
    \end{equation}
    Here we have $FI_{[0,t]} \in W^{-\delta,2}(0,T)$ for any $\delta>0$, and thus, thanks to~\cref{le:inversion}, we have $G_{\beta\beta'}[FI_{[0,t]}] \in W^{\frac{1-\beta\beta'}{2},2}(0,T)$.
    Moreover,~\cite[Proposition 2.1]{hitchhiker} implies that $W^{\frac{1-\beta\beta'}{2},2}(0,T) \subset W^{s_{\beta'},2}(0,T)$ since $1-\beta\beta' \geq 1-\beta'^2$. Thus,
    \begin{equation*}
        G_{\beta\beta'}[FI_{[0,t]}] \in W^{s_{\beta'},2}(0,T)
    \end{equation*}
    and hence, by linearity, it suffices to consider the terms
    \begin{equation*}
        \tilde{u}_z(0) + e^{-\beta\beta'\EX [X(z)X(0)]}\tilde{u}_z'(0)(G_{\beta\beta'}I_{[0,t]})(z)
    \end{equation*}
    arising from boundary conditions.
    For the initial value problem we have $\tilde{u}_z(0)=U_1$ and $e^{-\beta\beta'\EX [X(z)X(0)]}\tilde{u}_z'(0)=U_2$ from which the claim follows directly by the above argument (by replacing $F$ with a constant function 1). Similarly, with Neumann boundary conditions we may assume (by shifting with an arbitrary constant in the solution if necessary) that $\tilde{u}_z(0)=0$. Thus, this case follows directly by the above argument as well. With Dirichlet data we see that~\cref{eq:tilde_uz} at time $T$ becomes
    \begin{equation*}
        U_2 = \tilde{u}_z(T) = U_1 + e^{-\beta\beta'\EX [X(z)X(0)]}\tilde{u}_z'(0)(G_{\beta\beta'}I_{[0,T]})(z) + (G_{\beta\beta'}[F I_{[0,T]}])(z)
    \end{equation*}
    from which we get
    \begin{equation*}
        e^{-\beta\beta'\EX [X(z)X(0)]}\tilde{u}_z'(0) = \frac{U_2-U_1 - (G_{\beta\beta'}[F I_{[0,T]}])(z)}{(G_{\beta\beta'}I_{[0,T]})(z)}.
    \end{equation*}
    Inserting the above into \cref{eq:tilde_uz}, we are left to prove that
    \begin{equation*}
        (G_{\beta\beta'}I_{[0,t]})(z)\cdot\frac{U_2-U_1 - (G_{\beta\beta'}[F I_{[0,T]}])(z)}{(G_{\beta\beta'}I_{[0,T]})(z)} \in W^{s_{\beta'},2}(0,T).
    \end{equation*}
    Note that here
    \begin{equation*}
        \begin{split}
            (G_{\beta\beta'}I_{[0,t]})(z) & = \int_0^T |z-y|^{-\beta\beta'}e^{\beta\beta' g(z,y)}I_{[0,t]}(y)dy \\
                                        & = \int_0^t |z-y|^{-\beta\beta'}e^{\beta\beta' g(z,y)}dy.
        \end{split}
    \end{equation*}
    Since $g$ is a bounded function and $\beta\beta'<1$, it follows that, for any $t\in (0,T)$, $G_{\beta\beta'}I_{[0,t]}(z)$ is $1-\beta\beta'$ H\"older continuous. Thus, for any $\beta'\geq \beta$ we have $G_{\beta\beta'}I_{[0,t]} \in L^\infty([0,T]) \cap W^{s_{\beta'},2}(0,T)$. Indeed, for $\beta'>\beta$ this follows from $C^{1-\beta\beta'}(0,T)\subset  W^{s_{\beta'},2}(0,T)$ since
    \begin{equation*}
        1-\beta\beta' \geq 1-(\beta')^2
    \end{equation*}
    while for $\beta'=\beta$ this follows directly from the fact $I_{[0,t]} \in W_0^{-s_\beta,2}(0,T)$ and~\cref{le:inversion}.
    Moreover, for any $t\in (0,T)$ the function $G_{\beta\beta'}I_{[0,t]}$ is bounded from below, and hence 
    \begin{equation*}
        \frac{1}{G_{\beta\beta'}I_{[0,t]}} \in L^\infty([0,T]) \cap W^{s_{\beta'},2}(0,T) \quad \textnormal{for any } t\in(0,T).
    \end{equation*}
    Arguing similarly we also see that
    $U_2-U_1-(G_{\beta\beta'}FI_{[0,T]})(z) \in L^\infty([0,T]) \cap W^{s_{\beta'},2}(0,T)$ and hence the claim follows from~\cref{lemma:simple}. Treating the periodic case with similar arguments completes the whole proof.
\end{proof}
\begin{remark}
    As pointed out in~\cref{remark:comparison}, for the initial value problem or in the case of Neumann boundary conditions, choosing $\beta'=\beta$ in~\cref{thm:projection} yields precisely the solutions provided in~\cref{thm:wick-eqn-is-solvable}.  On the contrary, with Dirichlet boundary conditions we obtain that ``the projection'' of the Wick product
    \begin{equation*}
        \kappa \diamond \int_0^t e^{\diamond X_\beta(s)}ds
    \end{equation*}
    into $L_\beta^2(\Omega)$ leads to the term (on the $S$-transformed side)
    \begin{equation*}
        (G_{\beta^2}I_{[0,t]})(z)\cdot\frac{U_2-U_1 - (G_{\beta^2}FI_{[0,T]})(z)}{(G_{\beta^2}I_{[0,T]})(z)}
    \end{equation*}
    which is precisely what one would expect when taking the $S$-transform at $X_\beta(z)$ on $\kappa \diamond \int_0^t e^{\diamond X_\beta(s)}ds$ and using identities $S(X\diamond Y) = (SX)(SY)$ and $(SX^{\diamond (-1)}) = 1/(SX)$. This computation is of course only formal as $S$-transform at points $X_\beta(z)$ is not defined in the Hida distribution space $(\mathcal{S})_{-1}$.
\end{remark}

\bibliographystyle{plain}
\bibliography{1D-wick-non-wick}

\end{document}